\pdfoutput=1
\documentclass[12pt,a4paper]{article}
\usepackage[left=2.6cm,right=2.6cm,top=2.6cm,bottom=3cm,a4paper]{geometry}
\usepackage{amsmath,amsfonts,amssymb,amsthm}
\usepackage{pifont}
\usepackage{mathtools}
\usepackage{bm,bbm}
\usepackage[normalem]{ulem}
\usepackage{enumitem}
\usepackage{color,xcolor}
\usepackage{mathrsfs}
\usepackage{titlesec}
\usepackage{titletoc}
\usepackage{tocloft}
\usepackage{booktabs}
\usepackage{tabularx,makecell,multirow}
\usepackage{graphics,pstricks}
\usepackage{float}
\usepackage{subfig}
\usepackage{tikz}
\usetikzlibrary{positioning}
\usetikzlibrary{graphs}
\usetikzlibrary{graphs.standard}
\usepackage{algorithm}
\usepackage{algpseudocode}
\usepackage[matrix,arrow]{xy}
\usepackage{makeidx}
\usepackage[colorlinks=true,linkcolor=red]{hyperref}
\numberwithin{equation}{section}

\theoremstyle{plain}
\newtheorem{theorem}{Theorem}[section]
\newtheorem{lemma}[theorem]{Lemma}
\newtheorem{proposition}[theorem]{Proposition}

\newtheorem{question}{Question}
\newtheorem{conjecture}{Conjecture}

\theoremstyle{definition}
\newtheorem{definition}[theorem]{Definition}
\newtheorem{example}[theorem]{Example}

\theoremstyle{remark}
\newtheorem*{remark}{Remark}

\allowdisplaybreaks

\title{Classification of exchange relation planar algebras through sieving forest fusion graphs}
\author{Fan Lu \and Zhengwei Liu}
\date{}

\begin{document}

    \maketitle

    \begin{abstract}
        We suggest a classification scheme for subfactorizable fusion bialgebras, particularly for exchange relation planar algebras. This scheme begins by transforming infinite diagrammatic consistency equations of exchange relations into a finite set of algebraic equations of degree at most 3. We then introduce a key concept, the fusion graph of a fusion bialgebra, and prove that the fusion graph for any minimal projection is a forest if and only if the planar algebra has an exchange relation. For each fusion graph, the system of degree 3 equations reduces to linear and quadratic equations that are efficiently solvable. To deal with exponentially many fusion graphs, we propose two novel analytic criteria to sieve most candidates from being subfactor planar algebras.  Based on these results, we classify 5-dimensional subfactorizable fusion bialgebras with exchange relations. This scheme recovers the previous classification up to 4-dimension by Bisch, Jones, and the second author with quick proof. We developed a computer program to sieve forest fusion graphs using our criteria without solving the equations. The efficiency is 100\%, as all remaining graphs are realized by exchange relation planar algebras up to 5-dimension. 
    \end{abstract}

    \section{Introduction}

    Bisch and Jones initiated the classification of subfactor planar algebras of small dimensions generated by a non-trivial 2-box, namely a 4-valent vertex. The small dimension condition of 3-boxes ensures sufficiently nice skein relations to evaluate 4-valent planar algebras, so that the planar algebras can be classified by solving the consistency equations of these parameterized relations.  
    In particular, the Bisch-Jones classification of planar algebras up to 13 dimensional 3-boxes in \cite{Bisch2000Singly,Bisch2003Singly} is equivalent to the classification of singly generated exchange relation planar algebras. 
    The Yang-Baxter relation is a more comprehensive skein relation for 2-boxes, and the classification of singly generated Yang-Baxter relation planar algebras was achieved in \cite{Liu2016Yang}, including three distinct families:  BMW planar algebras \cite{Birman1989Braids,Murakami1987Kauffman,Wenzl1990Quantum}, Bisch-Jones planar algebras \cite{Bisch1997Algebras}, and another one-parameter family of planar algebras \cite{Liu2016Yang}. 

    Exchange relation planar algebras with two generators were classified in \cite{Liu2016Exchange}. Together with the two Temperley-Lieb elements of non-intersecting pairs of strings, the 2-box space is 4 dimensional and the 3-box space is at most 25 dimensional.
    It was observed that the exchange relation is determined by the structure of the 2-boxes, which includes trace, multiplication, convolution, adjoint, and contragredient, as illustrated in Figure \ref{figureintro}.

    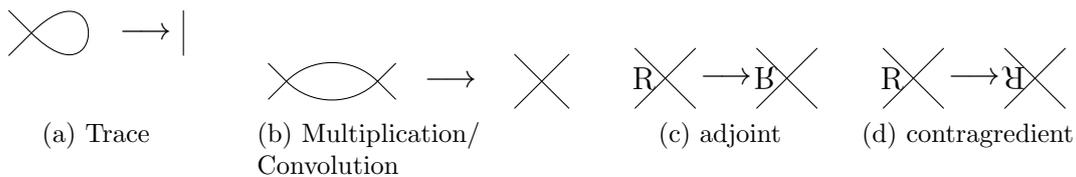
\begin{figure}[ht]
        \centering
        \subfloat[][Trace]{
            \begin{tikzpicture}
                \draw (0,0)  .. controls (1,1) and (1,-1) ..  (0,0);
                \draw[] (0,0) -- (-0.3,0.3);
                \draw[] (0,0) -- (-0.3,-0.3);
                \node[] (A) at (1.5,0) {$ \longrightarrow $};
                \draw[] (2,-0.3) -- (2,0.3);
            \end{tikzpicture}
        }
        \quad \quad
        \subfloat[][Multiplication/ \\ Convolution]{
            \begin{tikzpicture}[scale=1.2]
                \draw (0,0) -- (-0.2,0.2);
                \draw (0,0) -- (-0.2,-0.2);
                \draw (0,0) to[bend left=45] (1,0)--(1.2,0.2) ;
                \draw (0,0) to[bend right=45] (1,0)--(1.2,-0.2);
                \node[] (A) at (1.8,0) {$ \longrightarrow $};
                \draw[] (2.5,-0.3) -- (3.1,0.3);
                \draw[] (2.5,0.3) -- (3.1,-0.3);
            \end{tikzpicture}
        }
        \quad
        \subfloat[][adjoint]{
            \begin{tikzpicture}[scale=0.8]
                \draw (0,0) -- (1,1);
                \draw (1,0) -- (0,1);
                \node[] (l) at (0.15,0.5) {R};
                \node[] (A) at (1.5,0.5) {$ \longrightarrow $};
                \draw (2,0) -- (3,1);
                \draw (3,0) -- (2,1);
                \node[] (r) at (2.15,0.5) [yscale=-1]{R};
            \end{tikzpicture}
        }
        \quad
        \subfloat[][contragredient]{
            \begin{tikzpicture}[scale=0.8]
                \draw (0,0) -- (1,1);
                \draw (1,0) -- (0,1);
                \node[] (l) at (0.15,0.5) {R};
                \node[] (A) at (1.5,0.5) {$ \longrightarrow $};
                \draw (2,0) -- (3,1);
                \draw (3,0) -- (2,1);
                \node[] (r) at (2.15,0.5) {\rotatebox{180}{R}};
            \end{tikzpicture}
        }
        \caption{The structure of 2-boxes.}
        \label{figureintro}
    \end{figure}
    
    This structure of the abelian 2-box algebra has been axiomatized as a fusion bialgebra in \cite{Liu2021Fusion}, which was also motivated by the unitary categorification problem of fusion rings.

      \begin{definition}
        Let $ \mathcal{B} $ be a unital $*$-algebra over the complex field $ \mathbb{C} $. We define $ \mathcal{B} $ as a fusion bialgebra if $ \mathcal{B} $ has a $ \mathbb{R}_{\ge 0} $-basis $ B=\{x_{0}=1_{\mathcal{B}}, x_{1}, \cdots, x_{n}\} $ that satisfies the following conditions:
        \begin{enumerate}
            \item[(1)] ${x_0, . . . , x_n}$ is a linear basis over $ \mathbb{C} $;
            \item[(2)] $ x_jx_k=\sum\limits_{s=0}^{n} \widetilde{N}_{j,k}^{s}x_s $, $ \widetilde{N}_{j,k}^{s}\ge 0 $;
            \item[(3)]  there exists an involution $ * $ on $\{0, 1, 2, . . . , n\}$ such that $ x_{k}^{*}=x_{k^{*}} $ and $ \widetilde{N}_{j,k}^{0}=\delta_{j,k^{*}} $.
        \end{enumerate}
    \end{definition}

    A fusion bialgebra is called subfactorizable, if it can be realized as the 2-box space of an irreducible subfactor planar algebra. The classification of  2-dimensional subfactorizable fusion bialgebras is equivalent to the classification of Jones index \cite{Jones1983Index}. 
    For every $j$, we define the \textbf{weighted fusion graph} $\Gamma_j$ whose adjacency matrix $ \widetilde{N}_{j}^{} $ is defined by $ (\widetilde{N}_{j})_{s,k}:=\widetilde{N}_{j,k}^{s} $ (Def. \ref{def:cg}).
    A subfactorizable fusion bialgebra has an exchange relation if its $ \mathbb{R}_{\ge 0} $-basis serves as generators of the subfactor planar algebra and satisfies the following exchange relation (Def. \ref{Def: Exchange Relation}):
        \begin{align*}
            \begin{tikzpicture}
                \draw[] (0,0) -- (0.3,0.5);
                \draw[] (0.3,0) -- (0.6,0.5);
                \draw[] (0,0.5) -- (0.6,0);
                \node[] (A) at (1,0.25) {$\longrightarrow$};
            \end{tikzpicture}
            \begin{tikzpicture}
                \draw (0,0.5) .. controls (0,0.25) and (0.6,0.25) .. (0.6,0.5);
                \draw (0,0) .. controls (0,0.25) and (0.6,0.25) .. (0.6,0);
                \draw[] (0.3,0) -- (0.3,0.5);
                \node[] (A) at (1,0.25) {$+$};
            \end{tikzpicture} 
            \begin{tikzpicture}
                \draw[] (0,0.5) -- (0.3,0);
                \draw[] (0.3,0.5) -- (0.6,0);
                \draw[] (0,0) -- (0.6,0.5);
            \end{tikzpicture}.
        \end{align*}

    In this paper, we focus on the classification of subfactorizable fusion bialgebras with exchange relations, based on the following key result.

    \begin{theorem}
    A subfactorizable fusion bialgebras has an exchange relation if and only if every fusion graph $\Gamma_j$ is a forest.
    \end{theorem}

    This result establishes a surprising correspondence between exchange relations in skein theory and forests in graph theory, enabling us to propose the following more delicate classification scheme for subfactorizable fusion bialgebras with exchange relations.

    \subsection*{The classification scheme} 

    \begin{enumerate}[label={\textbf{Step} $\mathbf{\arabic*}$.},leftmargin=9ex]
    \item List complete Consistency Equations.  (Theorem \ref{thm:evathm}, \ref{thm:conthm}, \ref{thm:polyeq}).
    \item Simplify Consistency Equations for every forest fusion graph. (Theorem \ref{thm:parameter}).
    \item \textbf{(FF)} Enumerate all Forest Fusion graphs of a fusion bialgebra.
    \item \textbf{(RG)} Choose representative fusion graphs up to graph isomorphism.
    \item \textbf{(APC)} Sieve graphs that do not satisfy the Associative Positivity Criterion.
    \item \textbf{(FTPC)} Sieve graphs from the Free/Tensor Product of two planar algebras.
    \item Solve structure equations for each remaining fusion graph. If a solution exists, verify the unitarity condition for it to be a subfactor planar algebra.
    \end{enumerate}

    Step 1 and 2 involve theoretical results that we establish in this paper. Step 4 to 6 are sieving criteria that we propose. Step 3 to 7 are implemented by our automated computer program \textbf{FuBi}. The input of \textbf{FuBi} is an integer $ n $, and the output is the classification result of $ n $-dimensional subfactorizable fusion bialgebras with exchange relations.

    In step 1, we develop an evaluation algorithm and present complete consistency conditions for exchange relation planar algebras, and reduce them to a finite set of algebraic equations concerning parameters in exchange relations, termed structure equations. This process resembles finding a finite generator set for an abstractly defined polynomial ideal. The highest degree of structure equations is 3, with the number of structure equations scaling as $ O(n^6) $ and the number of variables as $ O(n^4) $, see Theorem \ref{thm:polyeq}. The computational complexity for computing the Gr\"obner basis of the structure equations grows double exponentially, which is a common challenge in classification problems. 

    In step 2, the parameters in exchange relations are determined to be $ 0 $ or $ \pm1 $, depending on the forest fusion graph. Consequently, the fusion coefficient $ \widetilde{N}_{j,k}^{s} $ is an alternating sum of the quantum dimensions of 2-box generators. Thus the structure equations are significantly simplified by the forest fusion graph. The highest degree of structure equations is reduced from 3 to 2, and most equations are linear equations. The number of variables is reduced from $ O(n^{4}) $ to $ n $. The essential variables are quantum dimensions of 2-box generators, which fully determine the exchange relations. The system of structure equations is decomposed into exponentially many simpler systems corresponding to forest fusion graphs. From the perspective of algebraic geometry, the forest structure is rigid enough to decompose a complicated algebraic variety into simpler subvarieties, so that each subvariety corresponds to an exchange relation planar algebra.

    In step 3, the fusion graph depends on the non-zero elements of fusion coefficients. We encode fusion coefficients into a binary string, with 1 representing a non-zero element, and traverse over all possible binary strings. The forest structure of fusion graphs can be detected easily according to binary strings, allowing us to enumerate all forest fusion graphs.

    In step 4, the permutation of 2-box generators leads to an isomorphism between fusion graphs. All possible permutations of 2-box generators form a group $ G $, which acts naturally on fusion coefficients and thus on binary strings. We select one representative forest fusion graph from each $ G $-orbit. 

    In step 5 and 6, we propose two novel analytic criteria for subfactorizable fusion bialgebras with exchange relations, the \textbf{associative positivity criterion (APC)} and the \textbf{free/tensor product criterion (FTPC)}, which are obstructions to solutions of the structure equations. 
    The APC follows from the sparsity of edges in a forest fusion graph and the positivity of fusion coefficients in the associativity equations. The FTPC follows from the fact that the set of exchange relation planar algebras is closed under free product construction, and the fact that the tensor product of an exchange relation planar algebra and a depth-2 subfactor planar algebra is also an exchange relation planar algebra.

    In step 7, we solve structure equations for each remaining fusion graph by computing its Gr\"obner basis.  If a solution exists, we have to check the unitary condition for subfactor planar algebras case by case.

    Following the classification scheme, we successfully classify subfactorizable fusion bialgebras with exchange relations for $ \dim \mathcal{P}_{2,\pm}=5 $, which is considered challenging in the previous method. Additionally, we provide a succinct proof of earlier classification results for $ \dim \mathcal{P}_{2,\pm}\leq 4$ in Section \ref{prevwork}. The sieving criteria is robust with $ 100\% $ efficiency for the classification up to 5-dimension, as the last step yields a solution for every remaining fusion graph.

    We illustrate the sieving process in Table \ref{tableintro} and state the classification result for the 5-dimensional case as follows. The number of fusion graphs passing each criterion and the associated time costs are detailed, which indicates that this classification scheme is highly efficient and practical. The unique fusion graph passing all criteria is 1-regular, indicating its origin from the group $ \mathbb{Z}_5 $. Luckily, this result is achieved through purely combinatorial approaches without the need to solve any equations directly.

    \begin{table}[htbp]
        \centering
        \begin{tabular}{crrr}
        \toprule
        Criteria & \# Graphs(case 1)  & \# Graphs(case 2) & \# Graphs(case 3) \\
        \midrule
         & $ 1048576 $ & $ 65536 $ & $ 1024 $ \\
        \textbf{FF} & $ 47381 $ & $ 4435 $ & $ 137 $ \\
        \textbf{RG} & $ 2137 $ & $ 1292 $ & $ 46 $ \\
        \textbf{APC} & $ 41 $ & $ 18 $ & $ 2 $ \\
        \textbf{FTPC} & $ 0 $ & $ 0 $ & $ 1 $ \\
        \hline
        Time(seconds) & 9.443 & 1.264 & 0.162 \\
        \bottomrule
        \end{tabular}
        \caption{Sieving process for $ \dim \mathcal{P}_{2,\pm} = 5 $.}
        \label{tableintro}
    \end{table}
    
 \begin{theorem}\label{fdthm2}
        Suppose $ \mathcal{B} $ is a 5-dimensional subfactorizable fusion bialgebra with exchange relations. Then $ \mathcal{B} $ arises from one of the following planar algebras:
        \begin{enumerate}
            \item[(1)] $ \mathcal{P}^{\mathbb{Z}_5} $;
            \item[(2)] $ \mathcal{P}*\mathcal{S} $ for two exchange relation planar algebras $ \mathcal{P} $ and $ \mathcal{S} $ with $ 2 \le \dim \mathcal{P}_{2,\pm} \le 4 $ and $ \dim \mathcal{P}_{2,\pm}+\dim \mathcal{S}_{2,\pm}=6 $.
        \end{enumerate}
    \end{theorem}

    \section{Preliminaries}

    \subsection{Subfactor planar algebra and exchange relation}\label{expadef}
        We refer the reader to \cite{Jones1999Planar,Jones2012Quadratic} for origins of subfactor planar algebras and \cite{Landau2002Exchange,Liu2016Exchange} for details of exchange relation planar algebras. We use the left side of an asymmetric letter instead of $ \star $ to indicate a distinguished interval that is always unshaded. The alternative shading on a planar tangle is determined by the distinguished interval. Since exchange relation planar algebras are generated by $ 2 $-boxes, the labelled planar tangle can be viewed as a labelled 4-valent planar graph by shrinking a 2-box to a crossing and labeling the element on the left side as follows. 
        \begin{align}
            \begin{tikzpicture}[inner sep=0pt,minimum size=6mm,baseline={([yshift=-0.5ex]current bounding box.center)}]
                \node (i) at (0,0) [rectangle, draw]{$ p $};
                \draw (-0.15,0.3) -- (-0.15,0.6);
                \draw (0.15,0.3) -- (0.15,0.6);
                \draw (0.15,-0.3) -- (0.15,-0.6);
                \draw (-0.15,-0.3) -- (-0.15,-0.6);
            \end{tikzpicture}
            \longrightarrow
            \begin{tikzpicture}[inner sep=0pt,minimum size=6mm,baseline={([yshift=-0.5ex]current bounding box.center)}]
                \draw (0,0) -- (1,1);
                \draw (0,1) -- (1,0);
                \node at (0.1,0.5) {$ p $};
            \end{tikzpicture}.
        \end{align} 
        We adopt both notations in the paper.

    \subsubsection{The structure of 2-boxes}

    The algebraic structure of 2-boxes consists of trace, adjoint, dual/contragredient, multiplication, and convolution. We sketch the structure of 2-boxes of a subfactor planar algebra $ \mathcal{P} $ in this subsection.

    We assume $ \mathcal{P}_{2,+} $ is abelian in the paper, so $ \mathcal{P}_{2,+} $ is a commutative $ C^{*} $-algebra. Let $ L_2 = \{p_i\}_{0\le i \le n} $ be pairwise orthogonal minimal projections that generate $ \mathcal{P}_{2,+} $. Let $ p_0 = \frac{1}{\delta} \begin{tikzpicture}[inner sep=0pt,minimum size=6mm,baseline={([yshift=-0.5ex]current bounding box.center)},scale=0.3]
        \draw (0,1) .. controls (0,0) and (1,0) .. (1,1);
        \draw (0,-1) .. controls (0,0) and (1,0) .. (1,-1);
    \end{tikzpicture} $ be the Jones projection, and $ \delta $ be the value of a circle.
    \begin{align}\label{circle}
        \begin{tikzpicture}[inner sep=0pt,minimum size=6mm,baseline={([yshift=-0.5ex]current bounding box.center)}]
            \draw (0,0) circle [radius=0.3];
        \end{tikzpicture} = \delta.
    \end{align}
    The resolution of identity is 
    \begin{align}\label{ROI}
        \begin{tikzpicture}[inner sep=0pt,minimum size=6mm,baseline={([yshift=-0.5ex]current bounding box.center)}]
            \draw (-0.15,-0.6) -- (-0.15,0.6);
            \draw (0.15,-0.6) -- (0.15,0.6);
        \end{tikzpicture} 
         = \sum\limits_{i=0}^{n}  
        \begin{tikzpicture}[inner sep=0pt,minimum size=6mm,baseline={([yshift=-0.5ex]current bounding box.center)}]
            \node (i) at (0,0) [rectangle, draw]{$ p_{i} $};
            \draw (-0.15,0.3) -- (-0.15,0.6);
            \draw (0.15,0.3) -- (0.15,0.6);
            \draw (0.15,-0.3) -- (0.15,-0.6);
            \draw (-0.15,-0.3) -- (-0.15,-0.6);
        \end{tikzpicture}.
    \end{align}
    The adjoint structure of 2-boxes is the $ * $ structure. Pictorially, it is a reflection across a horizontal line, for example, 
    \begin{align}\label{adjoint}
        \begin{tikzpicture}[inner sep=0pt,minimum size=6mm,baseline={([yshift=-0.5ex]current bounding box.center)}]
            \node (i) at (0,0) [rectangle, draw]{$ p^* $};
            \draw (-0.15,0.3) -- (-0.15,0.6);
            \draw (0.15,0.3) -- (0.15,0.6);
            \draw (0.15,-0.3) -- (0.15,-0.6);
            \draw (-0.15,-0.3) -- (-0.15,-0.6);
        \end{tikzpicture} := 
        \begin{tikzpicture}[inner sep=0pt,minimum size=6mm,baseline={([yshift=-0.5ex]current bounding box.center)}]
            \node (i) at (0,0) [rectangle, draw, yscale=-1]{$ p $};
            \draw (-0.15,0.3) -- (-0.15,0.6);
            \draw (0.15,0.3) -- (0.15,0.6);
            \draw (0.15,-0.3) -- (0.15,-0.6);
            \draw (-0.15,-0.3) -- (-0.15,-0.6);
        \end{tikzpicture}.
    \end{align}
    We define the dual of a 2-box $ p $ to be the $ 180^{\circ} $ rotation of itself, and denote it by $ \overline{p} $.
    \begin{align}\label{dual}
        \begin{tikzpicture}[inner sep=0pt,minimum size=6mm,baseline={([yshift=-0.5ex]current bounding box.center)}]
            \node (i) at (0,0) [rectangle, draw]{$ \overline{p} $};
            \draw (-0.15,0.3) -- (-0.15,0.6);
            \draw (0.15,0.3) -- (0.15,0.6);
            \draw (0.15,-0.3) -- (0.15,-0.6);
            \draw (-0.15,-0.3) -- (-0.15,-0.6);
        \end{tikzpicture} := 
        \begin{tikzpicture}[inner sep=0pt,minimum size=6mm,baseline={([yshift=-0.5ex]current bounding box.center)}]
            \node (i) at (0,0) [rectangle, draw, rotate = 180]{$ p $};
            \draw (-0.15,0.3) -- (-0.15,0.6);
            \draw (0.15,0.3) -- (0.15,0.6);
            \draw (0.15,-0.3) -- (0.15,-0.6);
            \draw (-0.15,-0.3) -- (-0.15,-0.6);
        \end{tikzpicture}.
    \end{align}
    Note that $ \overline{\overline{p}} = p $ and the dual of a minimal projection is also a minimal projection. The generators $ \{p_i\}_{0\le i \le n} $ are divided into two parts, elements that are dual of each other and self-dual elements. The dual operation extends naturally to subscript such that $ \overline{p_i} = p_{\overline{i}} $. 

    The irreducible condition, $ \dim \mathcal{P}_{1,\pm} = 1 $, means that adding a cap to a 2-box gives the scalar multiple of a strand. 
    \begin{align}
        \begin{tikzpicture}[inner sep=0pt,minimum size=6mm,baseline={([yshift=-0.5ex]current bounding box.center)}]
            \node (i) at (0,0) [rectangle, draw] {$ p_i $};
            \draw (-0.15,-0.3) -- (-0.15,-0.7);
            \draw (0.15,-0.3) -- (0.15,-0.7);
            \draw (-0.15,0.3) .. controls (-0.15,0.7) and (0.15,0.7) .. (0.15,0.3);
        \end{tikzpicture}  =c_i \ 
        \begin{tikzpicture}[inner sep=0pt,minimum size=6mm,baseline={([yshift=-0.5ex]current bounding box.center)}]
            \draw (-0.15,-0.7) -- (-0.15,0.3) .. controls (-0.15,0.7) and (0.15,0.7) .. (0.15,0.3) -- (0.15,-0.7);
        \end{tikzpicture},\quad
        \begin{tikzpicture}[inner sep=0pt,minimum size=6mm,baseline={([yshift=-0.5ex]current bounding box.center)}]
            \node (i) at (0,0) [rectangle, draw] {$ p_i $};
            \draw (-0.15,-0.3) -- (-0.15,-0.7);
            \draw (-0.15,0.3) -- (-0.15,0.7);
            \draw (0.15,0.3) .. controls (0.15,0.7) and (0.5,0.7) .. (0.5,0.3) -- (0.5,-0.3);
            \draw (0.15,-0.3) .. controls (0.15,-0.7) and (0.5,-0.7) .. (0.5,-0.3);
        \end{tikzpicture}  =d_i \ 
        \begin{tikzpicture}[inner sep=0pt,minimum size=6mm,baseline={([yshift=-0.5ex]current bounding box.center)}]
            \draw (0,-0.7) --  (0,0.7);
        \end{tikzpicture}.
    \end{align}
    Here $ d_i $ is called the quantum dimension of $ p_i $'s.

    We define multiplication in $ \mathcal{P}_{2,+} $ as vertical composition
    \begin{align}\label{multiplication}
        p_i p_j :=
        \begin{tikzpicture}[inner sep=0pt,minimum size=6mm,baseline={([yshift=-0.5ex]current bounding box.center)}]
            \node (i) at (0,0) [rectangle, draw]{$ p_i $};
            \node (j) at (0,1) [rectangle, draw] {$ p_j $};
            \draw (0.15,0.3) -- (0.15,0.7);
            \draw (0.15,-0.3) -- (0.15,-0.7);
            \draw (-0.15,0.3) -- (-0.15,0.7);
            \draw (-0.15,-0.3) -- (-0.15,-0.7);
            \draw (-0.15,1.3) -- (-0.15,1.7);
            \draw (0.15,1.3) -- (0.15,1.7);
        \end{tikzpicture}
        =\delta_{ij}
        \begin{tikzpicture}[inner sep=0pt,minimum size=6mm,baseline={([yshift=-0.5ex]current bounding box.center)}]
            \node (i) at (0,0) [rectangle, draw]{$ p_{i} $};
            \draw (-0.15,0.3) -- (-0.15,0.6);
            \draw (0.15,0.3) -- (0.15,0.6);
            \draw (0.15,-0.3) -- (0.15,-0.6);
            \draw (-0.15,-0.3) -- (-0.15,-0.6);
        \end{tikzpicture},
    \end{align}
    and convolution as horizontal composition
    \begin{align}\label{convolution}
        p_k*p_j :=
        \begin{tikzpicture}[inner sep=0pt,minimum size=6mm,baseline={([yshift=-0.5ex]current bounding box.center)}]
            \node (k) at (0,0) [rectangle, draw]{$ p_k $};
            \node (j) at (1,0) [rectangle, draw] {$ p_j $};
            \draw (0.15,0.3) .. controls (0.15,0.6) and (0.85,0.6) .. (0.85,0.3);
            \draw (0.15,-0.3) .. controls (0.15,-0.6) and (0.85,-0.6) .. (0.85,-0.3);
            \draw (-0.15,0.3) -- (-0.15,0.6);
            \draw (-0.15,-0.3) -- (-0.15,-0.6);
            \draw (1.15,0.3) -- (1.15,0.6);
            \draw (1.15,-0.3) -- (1.15,-0.6);
        \end{tikzpicture}
        = \sum_{i=0}^{n} {N}_{kj}^{i}
        \begin{tikzpicture}[inner sep=0pt,minimum size=6mm,baseline={([yshift=-0.5ex]current bounding box.center)}]
            \node (i) at (0,0) [rectangle, draw]{$ p_{i} $};
            \draw (-0.15,0.3) -- (-0.15,0.6);
            \draw (0.15,0.3) -- (0.15,0.6);
            \draw (0.15,-0.3) -- (0.15,-0.6);
            \draw (-0.15,-0.3) -- (-0.15,-0.6);
        \end{tikzpicture}. 
    \end{align}
    Thus $ c_0 = 1$ and $c_i = 0 $ for $ 1 \le i \le n $. Formula \ref{convolution} is called the fusion rule of $ \mathcal{P} $. The $ \{{N}_{kj}^{i}\}_{0\le k, j, i \le n} $ are non-negative real numbers by Schur product theorem in \cite{Liu2016Exchange} and we call them fusion coefficients. If we fix a subscript $ k $, $ {N}_{k\cdot}^{\cdot} $ is called the fusion matrix of $ p_k $.

    The inner product on $ \mathcal{P}_{n,\pm} $ is defined to be $ \langle a,b \rangle := \text{tr}_{n}(b^*a) $ where $ \text{tr}_{n}(p) = \begin{tikzpicture}[inner sep=0pt,minimum size=6mm,baseline={([yshift=-0.5ex]current bounding box.center)}]
        \node (i) at (0,0) [rectangle, draw]{$ p $};
        \draw (-0.15,0.3) -- (-0.15,0.6);
        \draw (0.15,0.3) -- (0.15,0.6);
        \draw (0.15,-0.3) -- (0.15,-0.6);
        \draw (-0.15,-0.3) -- (-0.15,-0.6);
        \draw (0.15,0.6) .. controls (0.15,0.8) and (0.55,0.8) .. (0.55,0.6);
        \draw (0.15,-0.6) .. controls (0.15,-0.8) and (0.55,-0.8) .. (0.55,-0.6);
        \draw (0.55,0.6) -- (0.55,-0.6);
        \draw (-0.15,0.6) .. controls (-0.15,1) and (0.85,1) .. (0.85,0.6);
        \draw (-0.15,-0.6) .. controls (-0.15,-1) and (0.85,-1) .. (0.85,-0.6);
        \draw (0.85,0.6) -- (0.85,-0.6);
        \node (u) at (0,0.45) {$ ... $};
        \node (d) at (0,-0.45) {$ ... $};
        \node (n) at (0.7,0) {$ n $};
    \end{tikzpicture} $ is the unnormalized Markov trace. The spherical property of subfactor planar algebras requires that the Markov trace is isotopy invariant on a sphere. 

    The string Fourier transform $ \mathfrak{F}_s: \mathcal{P}_{2,+} \to \mathcal{P}_{2,-} $ is a $ 90^{\circ} $ rotation. 
    \begin{align}
        \begin{tikzpicture}[inner sep=0pt,minimum size=6mm,baseline={([yshift=-0.5ex]current bounding box.center)}]
            \node (i) at (0,0) [rectangle, draw]{$ \mathfrak{F}_s(p) $};
            \draw (-0.15,0.3) -- (-0.15,0.6);
            \draw (0.15,0.3) -- (0.15,0.6);
            \draw (0.15,-0.3) -- (0.15,-0.6);
            \draw (-0.15,-0.3) -- (-0.15,-0.6);
        \end{tikzpicture} := 
        \begin{tikzpicture}[inner sep=0pt,minimum size=6mm,baseline={([yshift=-0.5ex]current bounding box.center)}]
            \node (i) at (0,0) [rectangle, draw, rotate = 90]{$ p $};
            \draw (-0.15,0.3) -- (-0.15,0.6);
            \draw (0.15,0.3) -- (0.15,0.6);
            \draw (0.15,-0.3) -- (0.15,-0.6);
            \draw (-0.15,-0.3) -- (-0.15,-0.6);
        \end{tikzpicture}.
    \end{align}
    It behaves like an ordinary Fourier transform. It maps the convolution in $ \mathcal{P}_{2,+} $ to the multiplication in $ \mathcal{P}_{2,-} $, i.e. $ \mathfrak{F}_s(x*y) = \mathfrak{F}_s(x)\mathfrak{F}_s(y) $. The Plancherel's formula $ \langle x,y \rangle = \langle \mathfrak{F}_s(x),\mathfrak{F}_s(y) \rangle $ holds.

    The consistency condition requires that different operations on $ 2 $-boxes should be compatible with each other, which gives restrictions on parameters.

    \begin{proposition}\label{basic_cons}
        The fusion coefficients satisfy Frobenius reciprocity 
        \begin{align}\label{frobenius}
            N_{kj}^{i}d_i = N_{j\overline{i}}^{\overline{k}}d_{\overline{k}} = N_{\overline{i}k}^{\overline{j}}d_{\overline{j}} = N_{\overline{j}\overline{k}}^{\overline{i}}d_{\overline{i}} = N_{\overline{k} i}^{j}d_j = N_{i\overline{j}}^{k}d_k.
        \end{align} 
    Moreover,
    \begin{align}
        N_{kj}^0 &= \delta_{k\overline{j}}d_k, \\
        N_{kj}^{i} &= N_{\overline{j}\overline{k}}^{\overline{i}}, \\
        d_i &= d_{\overline{i}}, \\
        \delta & = \sum_{i=0}^{n} d_i.
    \end{align}
    \end{proposition}
    \begin{proof}
        Isotopy and spherical property of a subfactor planar algebra gives 
    \begin{equation}
        \begin{aligned}
            \begin{tikzpicture}[inner sep=0pt,minimum size=6mm,baseline={([yshift=-0.5ex]current bounding box.center)}]
                \node (i) at (0,1) [rectangle, draw]{$ p_i $};
                \node (j) at (1,0) [rectangle, draw, rotate=90] {$ p_j $};
                \node (k) at (0,-1) [rectangle, draw]{$ p_k $};            
                \draw (0.15,0.7) .. controls +(down:5mm) and+(left:5mm) .. (0.7,0.15);
                \draw (0.15,-0.7) .. controls +(up:5mm) and+(left:5mm) .. (0.7,-0.15);
                \draw (1.3,0.15) .. controls +(right:5mm) and +(down:5mm) .. (1.8,0.7) -- (1.8,1.3) ;
                \draw (1.3,-0.15) .. controls +(right:5mm) and +(up:5mm) .. (1.8,-0.7) -- (1.8,-1.3) ;
                \draw (-0.15,0.7) -- (-0.15,-0.7);
                \draw (0.15,1.3) .. controls (0.15,1.8) and (1.8, 1.8) ..  (1.8,1.3);
                \draw (-0.15,1.3) ..controls (-0.15,1.8) and (-0.6, 1.8) .. (-0.6,1.3);
                \draw (0.15,-1.3) .. controls (0.15,-1.8) and (1.8,-1.8) .. (1.8,-1.3);
                \draw (-0.15,-1.3) ..controls (-0.15,-1.8) and (-0.6,-1.8) .. (-0.6,-1.3);
                \draw (-0.6,-1.3) -- (-0.6,1.3);
            \end{tikzpicture} &= \text{tr}_2((p_k*p_j)p_i) = \text{tr}_2((p_j*p_{\overline{i}})p_{\overline{k}}) = \text{tr}_2((p_{\overline{i}}*p_k)p_{\overline{j}}) \\
            &= \text{tr}_2((p_{\overline{j}}*p_{\overline{k}})p_{\overline{i}}) = \text{tr}_2((p_{\overline{k}}*p_i)p_j) = \text{tr}_2((p_i*p_{\overline{j}})p_k). \\
        \end{aligned}
    \end{equation}
    By definition, $ \text{tr}_2((p_k*p_j)p_i) = N_{kj}^{i}d_i\delta $, we obtain Frobenius reciprocity.

    It follows from above that $ \text{tr}_2((p_k*p_j)p_0) = N_{kj}^{0} $. The pictorial representation of $ p_0 $ provides another way to compute it.
    \begin{align}
        \text{tr}_2((p_k*p_j)p_0) = \frac{1}{\delta}
        \begin{tikzpicture}[inner sep=0pt,minimum size=6mm,baseline={([yshift=-0.5ex]current bounding box.center)}]
            \node (j) at (1,0) [rectangle, draw, rotate=90] {$ p_j $};
            \node (k) at (0,-1) [rectangle, draw]{$ p_k $};            
            \draw (0.15,0.7) .. controls +(down:5mm) and+(left:5mm) .. (0.7,0.15);
            \draw (0.15,0.7) .. controls (0.15,0.9) and (-0.15,0.9) .. (-0.15,0.7);
            \draw[] (0.15,1.3) .. controls (0.15,1.1) and (-0.15,1.1) .. (-0.15,1.3);
            \draw (0.15,-0.7) .. controls +(up:5mm) and+(left:5mm) .. (0.7,-0.15);
            \draw (1.3,0.15) .. controls +(right:5mm) and +(down:5mm) .. (1.8,0.7) -- (1.8,1.3) ;
            \draw (1.3,-0.15) .. controls +(right:5mm) and +(up:5mm) .. (1.8,-0.7) -- (1.8,-1.3) ;
            \draw (-0.15,0.7) -- (-0.15,-0.7);
            \draw (0.15,1.3) .. controls (0.15,1.8) and (1.8, 1.8) ..  (1.8,1.3);
            \draw (-0.15,1.3) ..controls (-0.15,1.8) and (-0.6, 1.8) .. (-0.6,1.3);
            \draw (0.15,-1.3) .. controls (0.15,-1.8) and (1.8,-1.8) .. (1.8,-1.3);
            \draw (-0.15,-1.3) ..controls (-0.15,-1.8) and (-0.6,-1.8) .. (-0.6,-1.3);
            \draw (-0.6,-1.3) -- (-0.6,1.3);
        \end{tikzpicture} = \delta_{k\overline{j}}d_k.
    \end{align}
    Thus $ N_{kj}^{0} = \delta_{k\overline{j}}d_k $.

    Take duals on both sides of Formula \ref{convolution} and compare coefficients of $ p_i $'s, we obtain $ N_{kj}^{i} = N_{\overline{j}\overline{k}}^{\overline{i}} $. 
    
    $ d_i = d_{\overline{i}} $ follows from two different ways to compute $ \text{tr}_2(p_i) $.
    \begin{align}
        \text{tr}_2(p_i) = 
        \begin{tikzpicture}[inner sep=0pt,minimum size=6mm,baseline={([yshift=-0.5ex]current bounding box.center)}]
            \node (i) at (0,0) [rectangle, draw] {$ p_i $};
            \draw (0.15,0.3) .. controls (0.15,0.7) and (0.5,0.7) .. (0.5,0.3) -- (0.5,-0.3);
            \draw (0.15,-0.3) .. controls (0.15,-0.7) and (0.5,-0.7) .. (0.5,-0.3);
            \draw (-0.15,0.3) .. controls (-0.15,0.7) and (-0.5,0.7) .. (-0.5,0.3) -- (-0.5,-0.3);
            \draw (-0.15,-0.3) .. controls (-0.15,-0.7) and (-0.5,-0.7) .. (-0.5,-0.3);
        \end{tikzpicture} = d_i\delta = d_{\overline{i}}\delta.
    \end{align}
    Take trace on both sides of Equation \ref{ROI}, and we obtain $ \delta^{2} = \sum_{i=0}^{n} d_i\delta $.
    \end{proof} 
    \begin{remark}
        These are called basic consistency equations on subfactor planar algebras.
    \end{remark}

    \subsubsection{Exchange relation}
    The exchange relation is determined by the algebraic structure of 2-boxes.

    \begin{definition}\label{Def: Exchange Relation}
        Suppose $ \mathcal{P} $ is an irreducible subfactor planar algebra. If $ \mathcal{P} $ is generated by $ \mathcal{P}_2 $ with the exchange relations 
        \begin{align}\label{exr1}
            \begin{tikzpicture}[inner sep=0pt,minimum size=6mm,baseline={([yshift=-0.5ex]current bounding box.center)}]
                \node (i) at (0,0) [rectangle, draw] {$ p_i $};
                \draw (-0.15,-0.3) -- (-0.15,-0.7);
                \draw (0.15,-0.3) -- (0.15,-0.7);
                \draw (-0.15,0.3) .. controls (-0.15,0.7) and (0.15,0.7) .. (0.15,0.3);
            \end{tikzpicture}  =c_i \ 
            \begin{tikzpicture}[inner sep=0pt,minimum size=6mm,baseline={([yshift=-0.5ex]current bounding box.center)}]
                \draw (-0.15,-0.7) -- (-0.15,0.3) .. controls (-0.15,0.7) and (0.15,0.7) .. (0.15,0.3) -- (0.15,-0.7);
            \end{tikzpicture},\quad
            \begin{tikzpicture}[inner sep=0pt,minimum size=6mm,baseline={([yshift=-0.5ex]current bounding box.center)}]
                \node (i) at (0,0) [rectangle, draw] {$ p_i $};
                \draw (-0.15,-0.3) -- (-0.15,-0.7);
                \draw (-0.15,0.3) -- (-0.15,0.7);
                \draw (0.15,0.3) .. controls (0.15,0.7) and (0.5,0.7) .. (0.5,0.3) -- (0.5,-0.3);
                \draw (0.15,-0.3) .. controls (0.15,-0.7) and (0.5,-0.7) .. (0.5,-0.3);
            \end{tikzpicture}  =d_i \ 
            \begin{tikzpicture}[inner sep=0pt,minimum size=6mm,baseline={([yshift=-0.5ex]current bounding box.center)}]
                \draw (0,-0.7) --  (0,0.7);
            \end{tikzpicture},
        \end{align}
        \begin{align}\label{exr2}
            \begin{tikzpicture}[inner sep=0pt,minimum size=6mm,baseline={([yshift=-0.5ex]current bounding box.center)}]
            \node (i) at (0,1) [rectangle, draw]{$ p_i $};
            \node (j) at (1,0) [rectangle, draw, rotate=90] {$ p_j $};            
            \draw (0.15,0.7) .. controls +(down:5mm) and+(left:5mm) .. (0.7,0.15);
            \draw (0.15,-0.7) .. controls +(up:5mm) and+(left:5mm) .. (0.7,-0.15);
            \draw (1.3,0.15) .. controls +(right:5mm) and +(down:5mm) .. (1.8,0.7) -- (1.8,1.6) ;
            \draw (1.3,-0.15) .. controls +(right:5mm) and +(up:5mm) .. (1.8,-0.7) ;
            \draw (-0.15,0.7) -- (-0.15,-0.7);
            \draw (0.15,1.3) -- (0.15,1.6);
            \draw (-0.15,1.3) -- (-0.15,1.6);
            \end{tikzpicture}
             =\sum\limits_{s,t=0}^{n} a_{st}^{ij} \ 
            \begin{tikzpicture}[inner sep=0pt,minimum size=6mm,baseline={([yshift=-0.5ex]current bounding box.center)}]
                \node (s) at (0,1) [rectangle, draw]{$ p_s $};
                \node (t) at (0,0) [rectangle, draw] {$ p_t $};
                \draw (-0.15,0.7) -- (-0.15,0.3);
                \draw (0.15,1.3) -- (0.15,1.6);
                \draw (-0.15,1.3) -- (-0.15,1.6);
                \draw (0.15,-0.3) -- (0.15,-0.7);
                \draw (-0.15,-0.3) -- (-0.15,-0.7);
                \draw (0.15,0.7) .. controls (0.3,0.5) and(0.7,0.5) .. (0.7,1.6);
                \draw (0.15,0.3) .. controls (0.3,0.5) and(0.7,0.5) .. (0.7,-0.7);
            \end{tikzpicture}
             +\sum\limits_{\ell,m=0}^{n} b_{\ell m}^{ij} \ 
            \begin{tikzpicture}[inner sep=0pt,minimum size=6mm,baseline={([yshift=-0.5ex]current bounding box.center)}]
            \node (t) at (0,-1) [rectangle, draw]{$ p_m $};
            \node (s) at (1,0) [rectangle, draw, rotate=90] {$ p_{\ell} $};            
            \draw (0.15,0.7) .. controls +(down:5mm) and+(left:5mm) .. (0.7,0.15);
            \draw (0.15,-0.7) .. controls +(up:5mm) and+(left:5mm) .. (0.7,-0.15);
            \draw (1.3,0.15) .. controls +(right:5mm) and +(down:5mm) .. (1.8,0.7)  ;
            \draw (1.3,-0.15) .. controls +(right:5mm) and +(up:5mm) .. (1.8,-0.7) -- (1.8,-1.6) ;
            \draw (-0.15,0.7) -- (-0.15,-0.7);
            \draw (0.15,-1.3) -- (0.15,-1.6);
            \draw (-0.15,-1.3) -- (-0.15,-1.6);
            \end{tikzpicture},
        \end{align}
        where $ \{p_i\}_{0\le i \le n} $ is a generator set of $ \mathcal{P}_{2,+} $ and $ c_i, d_i, a_{st}^{ij}, b_{\ell m}^{ij} \in \mathbb{C} $, then $ \mathcal{P} $ is called an exchange relation planar algebra. 
    \end{definition}

    We make the conventions for simplicity.

    \begin{align}
        T(i,j,k) := \begin{tikzpicture}[inner sep=0pt,minimum size=6mm,baseline={([yshift=-0.5ex]current bounding box.center)}]
            \node (i) at (0,1) [rectangle, draw]{$ p_i $};
            \node (j) at (1,0) [rectangle, draw, rotate=90] {$ p_j $};
            \node (k) at (0,-1) [rectangle, draw]{$ p_k $};            
            \draw (0.15,0.7) .. controls +(down:5mm) and+(left:5mm) .. (0.7,0.15);
            \draw (0.15,-0.7) .. controls +(up:5mm) and+(left:5mm) .. (0.7,-0.15);
            \draw (1.3,0.15) .. controls +(right:5mm) and +(down:5mm) .. (1.8,0.7) -- (1.8,1.6) ;
            \draw (1.3,-0.15) .. controls +(right:5mm) and +(up:5mm) .. (1.8,-0.7) -- (1.8,-1.6) ;
            \draw (-0.15,0.7) -- (-0.15,-0.7);
            \draw (0.15,1.3) -- (0.15,1.6);
            \draw (-0.15,1.3) -- (-0.15,1.6);
            \draw (0.15,-1.3) -- (0.15,-1.6);
            \draw (-0.15,-1.3) -- (-0.15,-1.6);
        \end{tikzpicture},\  & \ 
        W(i,j,k) := \begin{tikzpicture}[inner sep=0pt,minimum size=6mm,baseline={([yshift=-0.5ex]current bounding box.center)}]
            \node (i) at (0,1) [rectangle, draw, rotate=90]{$ p_i $};
            \node (j) at (-1,0) [rectangle, draw] {$ p_j $};
            \node (k) at (0,-1) [rectangle, draw, rotate=90]{$ p_k $};            
            \draw (-0.15,0.7) .. controls +(down:5mm) and+(right:5mm) .. (-0.7,0.15);
            \draw (-0.15,-0.7) .. controls +(up:5mm) and+(right:5mm) .. (-0.7,-0.15);
            \draw (0.15,0.7) -- (0.15,-0.7);
            \draw (0.15,1.3) -- (0.15,1.6);
            \draw (-0.15,1.3) -- (-0.15,1.6);
            \draw (0.15,-1.3) -- (0.15,-1.6);
            \draw (-0.15,-1.3) -- (-0.15,-1.6);
            \draw (-1.15,-1.6) -- (-1.15,-0.3);
            \draw (-1.15,1.6) -- (-1.15,0.3);
        \end{tikzpicture}, \\
        E_1(i,j) := \begin{tikzpicture}[inner sep=0pt,minimum size=6mm,baseline={([yshift=-0.5ex]current bounding box.center)}]
            \node (i) at (0,1) [rectangle, draw]{$ p_i $};
            \node (j) at (1,0) [rectangle, draw, rotate=90] {$ p_j $};
            \draw (0.15,0.7) .. controls +(down:5mm) and+(left:5mm) .. (0.7,0.15);
            \draw (0.15,-0.7) .. controls +(up:5mm) and+(left:5mm) .. (0.7,-0.15);
            \draw (1.3,0.15) .. controls +(right:5mm) and +(down:5mm) .. (1.8,0.7) -- (1.8,1.6) ;
            \draw (1.3,-0.15) .. controls +(right:5mm) and +(up:5mm) .. (1.8,-0.7) ;
            \draw (-0.15,0.7) -- (-0.15,-0.7);
            \draw (0.15,1.3) -- (0.15,1.6);
            \draw (-0.15,1.3) -- (-0.15,1.6);
        \end{tikzpicture}, \quad
        E_2(i,j) :=& \begin{tikzpicture}[inner sep=0pt,minimum size=6mm,baseline={([yshift=-0.5ex]current bounding box.center)}]
            \node (s) at (0,1) [rectangle, draw]{$ p_i $};
            \node (t) at (0,0) [rectangle, draw] {$ p_j $};
            \draw (-0.15,0.7) -- (-0.15,0.3);
            \draw (0.15,1.3) -- (0.15,1.6);
            \draw (-0.15,1.3) -- (-0.15,1.6);
            \draw (0.15,-0.3) -- (0.15,-0.7);
            \draw (-0.15,-0.3) -- (-0.15,-0.7);
            \draw (0.15,0.7) .. controls (0.3,0.5) and(0.7,0.5) .. (0.7,1.6);
            \draw (0.15,0.3) .. controls (0.3,0.5) and(0.7,0.5) .. (0.7,-0.7);
        \end{tikzpicture}, \quad
        E_3(i,j) := \begin{tikzpicture}[inner sep=0pt,minimum size=6mm,baseline={([yshift=-0.5ex]current bounding box.center)}]
            \node (t) at (0,-1) [rectangle, draw]{$ p_j $};
            \node (s) at (1,0) [rectangle, draw, rotate=90] {$ p_i $};
            \draw (0.15,0.7) .. controls +(down:5mm) and+(left:5mm) .. (0.7,0.15);
            \draw (0.15,-0.7) .. controls +(up:5mm) and+(left:5mm) .. (0.7,-0.15);
            \draw (1.3,0.15) .. controls +(right:5mm) and +(down:5mm) .. (1.8,0.7)  ;
            \draw (1.3,-0.15) .. controls +(right:5mm) and +(up:5mm) .. (1.8,-0.7) -- (1.8,-1.6) ;
            \draw (-0.15,0.7) -- (-0.15,-0.7);
            \draw (0.15,-1.3) -- (0.15,-1.6);
            \draw (-0.15,-1.3) -- (-0.15,-1.6);
        \end{tikzpicture}.
    \end{align}

    \begin{definition}
        A biprojection is a projection $ Q $ in $ \mathcal{P}_{2,+} $ satisfies $ Q*Q = cQ $ where $ c $ is a scalar.
    \end{definition}

    \begin{proposition}
        A biprojection $ Q $ satisfies exchange relation. Specifically, $ E_1(Q,Q) = E_3(Q,Q) $.
    \end{proposition}
    \begin{proof}
        Let $ x = E_1(Q,Q) = E_3(Q,Q) $, then it is easy to check $ \text{tr}_3(x^*x) = 0 $. By positivity of trace, we have $ x = 0 $.
    \end{proof}

    \begin{proposition}
        $ COB = \{T(i,j,k): N_{kj}^{i} \ne 0, 0 \le i, j, k \le n \} $ is an orthogonal set of the 3-box space $ \mathcal{P}_{3,+} $.
    \end{proposition}
    \begin{proof}
        We have 
    \begin{equation}
    \begin{aligned}
       \langle T(i,j,k), T(i',j',k') \rangle &= \text{tr}_3(T(i',j',k')^*T(i,j,k)) \\
       &= \begin{tikzpicture}[inner sep=0pt,minimum size=6mm,baseline={([yshift=-0.5ex]current bounding box.center)}]
        \node (i) at (0,1) [rectangle, draw]{$ p_{k^{'}} $};
        \node (j) at (1,0) [rectangle, draw, rotate=-90] {$ p_{j^{'}} $};
        \node (k) at (0,-1) [rectangle, draw]{$ p_{i^{'}} $};
        \draw (0.15,0.7) .. controls +(down:5mm) and+(left:5mm) .. (0.7,0.15);
        \draw (0.15,-0.7) .. controls +(up:5mm) and+(left:5mm) .. (0.7,-0.15);
        \draw (1.3,0.15) .. controls +(right:5mm) and +(down:5mm) .. (1.8,0.7) -- (1.8,1.6) ;
        \draw (1.3,-0.15) .. controls +(right:5mm) and +(up:5mm) .. (1.8,-0.7) -- (1.8,-1.6) ;
        \draw (-0.15,0.7) -- (-0.15,-0.7);
        \draw (0.15,1.3) -- (0.15,1.6);
        \draw (-0.15,1.3) -- (-0.15,1.6);
        \draw (0.15,-1.3) -- (0.15,-1.6);
        \draw (-0.15,-1.3) -- (-0.15,-1.6);
        \node (i) at (0,4.2) [rectangle, draw]{$ p_i $};
        \node (j) at (1,3.2) [rectangle, draw, rotate=90] {$ p_j $};
        \node (k) at (0,2.2) [rectangle, draw]{$ p_k $};
        \draw (0.15,3.9) .. controls +(down:5mm) and+(left:5mm) .. (0.7,3.35);
        \draw (0.15,2.5) .. controls +(up:5mm) and+(left:5mm) .. (0.7,3.05);
        \draw (1.3,3.35) .. controls +(right:5mm) and +(down:5mm) .. (1.8,3.9) -- (1.8,4.8) ;
        \draw (1.3,3.05) .. controls +(right:5mm) and +(up:5mm) .. (1.8,2.5) -- (1.8,1.6) ;
        \draw (-0.15,3.9) -- (-0.15,2.5);
        \draw (0.15,4.5) -- (0.15,4.8);
        \draw (-0.15,4.5) -- (-0.15,4.8);
        \draw (0.15,1.9) -- (0.15,1.6);
        \draw (-0.15,1.9) -- (-0.15,1.6);
        \draw (1.8,4.8) .. controls (1.8,5) and (2,5) .. (2,4.8) -- (2,-1.6);
        \draw (1.8,-1.6) .. controls (1.8,-1.8) and (2,-1.8) .. (2,-1.6);
        \draw (0.15,4.8) .. controls (0.15,6) and (2.5,6) .. (2.5,4.8) -- (2.5,-1.6);
        \draw (0.15,-1.6) .. controls (0.15,-2.8) and (2.5,-2.8) .. (2.5,-1.6);
        \draw (-0.15,4.8) .. controls (-0.15,6.5) and (2.8,6.5) .. (2.8,4.8) -- (2.8,-1.6);
        \draw (-0.15,-1.6) .. controls (-0.15,-3.3) and (2.8,-3.3) .. (2.8,-1.6);
    \end{tikzpicture}
    =\delta_{ii^{'}}\delta_{jj^{'}}\delta_{kk^{'}}{N}_{kj}^{i}d_i\delta.
    \end{aligned}
    \end{equation}
    It follows that different $ T $-vectors are orthogonal.
    Note that $ \delta > 0  $ and $ d_i = \frac{1}{\delta}\text{tr}_2(p_i) > 0 $, so $ T(i,j,k) = 0 \Longleftrightarrow N_{kj}^{i} = 0 $.
    \end{proof}

    \begin{proposition}\label{OBP3}
        $ COB = \{T(i,j,k): N_{kj}^{i} \ne 0 \} $ is an orthogonal basis of $ \mathcal{P}_{3,+} $.
   \end{proposition}
   \begin{proof}
       $ \mathcal{P}_{3,+} $ is linearly spanned by elements with no internal faces according to Theorem 1 in \cite{Landau2002Exchange}, i.e. $ E_1, E_2, E_3 $-vectors. Because of the resolution of identity, $ E_1, E_2, E_3 $-vectors are  a sum of $ T $-vectors in $ \mathcal{P}_{3,+} $.
   \end{proof}
   \begin{remark}
       More generally, $ COB $ is an orthogonal basis of $ \mathcal{P}_{3,+} $ if $ \mathcal{P} $ has Yang-Baxter relations.
   \end{remark}

    \begin{proposition}
        Based on our choice of the generator set $ L_2 $, The parameters $ \{a_{st}^{ij}\} $, $ \{b_{\ell m}^{ij}\} $ in Exchange Relation \ref{exr2} satisfy
        \begin{align}
            a_{ik}^{ij} + b_{jk}^{ij} &= 1, \\
            a_{sk}^{ij} + b_{lk}^{ij} &= 0 \quad (s\ne i\  \text{or}\  \ell\ne j),\label{labeq}
        \end{align}
        for non-zero basis vectors $ T(i,j,k) $ and $ T(s,l,k) $.
    \end{proposition}
    
    \begin{proof}
        Write $ E_1, E_2, E_3 $-vectors as a sum of basis vectors in Exchange Relation \ref{exr2} and compare the coefficients of basis vectors in $ \mathcal{P}_{3,+} $.
    \end{proof}

    \subsection{Fusion bialgebra}

    Fusion bialgebra is an axiomatization of the structure of 2-boxes of subfactor planar algebras. The original definition of a fusion bialgebra in \cite{Liu2021Fusion} is a quintuple, we give an equivalent concise definition in this section to simplify notations.

    Let $ \mathbb{R}_{\ge 0} $ be the set of non-negative real numbers.
    \begin{definition}
        Let $ \mathcal{B} $ be a unital $*$-algebra over the complex field $ \mathbb{C} $. We say $ \mathcal{B} $ has a $ \mathbb{R}_{\ge 0} $-basis $ B=\{x_{0}=1_{\mathcal{B}}, x_{1}, \cdots, x_{n}\} $ if
        \begin{enumerate}
            \item[(1)] ${x_0, . . . , x_n}$ is a linear basis over $ \mathbb{C} $;
            \item[(2)] $ x_jx_k=\sum\limits_{s=0}^{n} \widetilde{N}_{j,k}^{s}x_s $, $ \widetilde{N}_{j,k}^{s}\in \mathbb{R}_{\ge 0} $;
            \item[(3)]  there exists an involution $ * $ on $\{0, 1, 2, . . . , n\}$ such that $ x_{k}^{*}=x_{k^{*}} $ and $ \widetilde{N}_{j,k}^{0}=\delta_{j,k^{*}} $.
        \end{enumerate}
    \end{definition}

    The definition of $ \mathbb{R}_{\ge 0} $-basis is similar to that of fusion ring, except that the fusion coefficients $ \mathbb{Z}_{+} $ are replaced by $ \mathbb{R}_{\ge 0} $.

    \begin{definition}
        Suppose $ \mathcal{A} $ and $ \mathcal{B} $ are two finite dimensional $C^*$-algebras with faithful traces $d$ and $ \tau $ respectively, $ \mathcal{A} $ is commutative, and $\mathscr{F} : \mathcal{A} \to \mathcal{B} $ is a unitary transformation preserving $2$-norms. We call the quintuple $(\mathcal{A}, \mathcal{B}, \mathscr{F}, d, \tau)$ a fusion bialgebra, if the following conditions hold:
        \begin{enumerate}
            \item[(1)] Schur Product: For operators $x, y \ge 0$ in $ \mathcal{A} $, $x * y := \mathscr{F}^{-1}(\mathscr{F}(x)\mathscr{F}(y)) \ge 0$ in $ \mathcal{A} $.
            \item[(2)] Modular Conjugation: The map $ J(x) := \mathscr{F}^{-1}(\mathscr{F}(x)^*)$ is an anti-linear, $*$-isomorphism on $ \mathcal{A} $. 
            \item[(3)] Jones Projection: The operator $ \mathscr{F}^{-1}(1) $ is a positive multiple of a minimal projection in $ \mathcal{A} $.
        \end{enumerate}
        Furthermore, if $ \mathscr{F}^{-1}(1) $ is a minimal projection and $ d(\mathscr{F}^{-1}(1))=1 $, then we call the fusion bialgebra canonical.
    \end{definition}

    Given a fusion bialgebra $(\mathcal{A}, \mathcal{B}, \mathscr{F}, d, \tau)$, then $(\mathcal{A}, \mathcal{B}, \lambda_{1}^{\frac{1}{2}}\lambda_{2}^{-\frac{1}{2}}\mathscr{F}, \lambda_1 d, \lambda_2 \tau)$ is also a fusion bialgebra with $ \lambda_1, \lambda_2 > 0 $ by definition. This is called \textbf{gauge transformation}. Therefore, any fusion bialgebra is equivalent to a canonical one up to a gauge transformation. 
    
    The definition of fusion bialgebra is a direct translation of properties of the structure of 2-boxes of subfactor planar algebras. It is proved that fusion bialgebra and $ \mathbb{R}_{\ge 0} $-basis are equivalent.

    \begin{theorem}[Liu-Palcoux-Wu, Theorem 2.17 in \cite{Liu2021Fusion}]\label{lpw}
        If $(\mathcal{A}, \mathcal{B}, \mathscr{F}, d, \tau)$ is a fusion bialgebra, then $ \mathcal{B} $ has a unique $ \mathbb{R}_{\ge 0} $-basis $ B=\{x_{0}, x_{1}, \cdots, x_{m} \}$, such that $\mathscr{F}^{-1}(x_j)$ are multiples of minimal projections of $ \mathcal{A} $. Moreover, $ B $ is invariant under the gauge transformation. Conversely, any $ C^* $-algebra $ \mathcal{B} $ with a $ \mathbb{R}_{\ge 0} $-basis $ B $ can be extended to a canonical fusion bialgebra, such that $\mathscr{F}^{-1}(x_j)$ are multiples of minimal projections of $ \mathcal{A} $.
    \end{theorem}
    
    Hence, we adopt an equivalent definition of fusion bialgebra.
    \begin{definition}
        Let $ \mathcal{B} $ be a unital $*$-algebra over the complex field $ \mathbb{C} $. We say $ \mathcal{B} $ is a fusion bialgebra if $ \mathcal{B} $ has a $ \mathbb{R}_{\ge 0} $-basis.
        
    \end{definition}

    We reformulate Theorem 2.24 in \cite{Liu2021Fusion} that asserts $ (\mathcal{P}_{2,+}, \mathcal{P}_{2,-}, \mathfrak{F}_s, \text{tr}_{2,+}, \text{tr}_{2,-}) $ is a fusion bialgebra.

    \begin{theorem}[Theorem 2.24 in \cite{Liu2021Fusion}]\label{fub}
        Suppose $ \mathcal{P} $ is a subfactor planar algebra. If $ \mathcal{P}_{2,+} $ is abelian, then $ \mathcal{P}_{2,+} $ is a fusion bialgebra with convolution and dual as involution.
    \end{theorem}
    \begin{proof}
        Take $ x_j = \sqrt{\frac{\delta}{d_j}} p_j $ and $ x_{j^{*}} = \sqrt{\frac{\delta}{d_{\overline{j}}}} p_{\overline{j}} $ and let $ \widetilde{N}_{jk}^{s} = \frac{N_{jk}^{s}\sqrt{\delta d_s}}{\sqrt{d_j d_k}} $. Then we have
        \begin{align*}
            x_j*x_k &= \frac{\delta}{\sqrt{d_j d_k}} p_j * p_k  \\
            &= \sum_{s=0}^{n} \frac{\delta}{\sqrt{d_j d_k}} N_{jk}^{s} p_s \\
            &= \sum_{s=0}^{n} \widetilde{N}_{jk}^{s}x_s.
        \end{align*}
        Moreover, $ \widetilde{N}_{jk}^{s}\in \mathbb{R}_{\ge 0} $, $ x_0 $ is the identity, $ \overline{x_j} = x_{\overline{j}} $, and $ \widetilde{N}_{jk}^{0} = \delta_{j\overline{k}} $. Hence, $ \{x_0,x_1,\dots, x_n\} $ is a $ \mathbb{R}_{\ge 0} $-basis of $ \mathcal{P}_{2,+} $ and $ \mathcal{P}_{2,+} $ is a fusion bialgebra.
    \end{proof}

    \begin{definition}
        A fusion bialgebra $ \mathcal{B} $ is called subfactorizable if it can be realized as the 2-box space $ \mathcal{P}_{2,+} $ of an irreducible subfactor planar algebra $ \mathcal{P} $.
        A subfactorizable fusion bialgebra has exchange relations if it originates from an exchange relation planar algebra, i.e. its $ \mathbb{R}_{\ge 0} $-basis satisfies Exchange Relation \ref{exr2}.
    \end{definition}
    \begin{remark}
        When $ \mathcal{B} $ is subfactorizable, we may use minimal projections of $ \mathcal{P}_{2,+} $ as a generator set instead of $ \mathbb{R}_{\ge 0} $-basis.
    \end{remark}
    \begin{remark}
        If a fusion ring is the Grothendieck ring of a unitary fusion category, it can be realized as the 2-box space of a subfactor planar algebra using the quantum double construction, such that the ring multiplication is implemented by the convolution of 2-boxes. Thus categorification implies subfactorization. It is interesting to find subfactorizable fusion ring that cannot be categorified. A potential candidate is $ \mathfrak{F}_{210} $ (see page 55 in \cite{Liu2021Fusion}). Although it is not categorical, whether it is subfactorizable remains an open question.
    \end{remark}

    \section{Evaluation and Consistency}

    Let $ L_2 $ be a 2-box generator set and $ \mathcal{P}(L_2) $ be the universal planar algebra. Its vector space $ \mathcal{P}(L_2)_{n,\pm} $ is the linear span of labelled $ 2n $-planar tangles, i.e., planar tangles with $ 2n $ boundary points whose input discs are labelled by elements in $ L_2 $. The main task to construct a planar algebra from a universal planar algebra is to define a positive semi-definite partition function $ Z $ for the universal planar algebra. Skein theory is a systematic approach to derive a partition function from suitable skein relations. In this process, three fundamental problems are
    \begin{enumerate}[label={$(\arabic*)$}]
        \item Evaluation: Every 0-box can be evaluated to a scalar. Moreover, each $ n $-box space is supposed to be finite dimensional.
        \item Consistency: Different evaluation paths give the same scalar. Then the partition function of the 0-box is defined to be this scalar. 
        \item Unitarity: The partition function derived from the evaluation is positive semi-definite with respect to the involution. Then the quotient of the universal planar algebra by the null space of the partition function is a subfactor planar algebra.
    \end{enumerate}

    We consider the case that $ \mathcal{P}(L_2) $ has exchange relations.
    Landau proved that each $ n $-box space is finite dimensional by Theorem 1 in \cite{Landau2002Exchange}. The second author gave an estimate on the upper bound of the dimensions that $ \dim \mathcal{P}_{n+1} \le (\dim \mathcal{P}_n)^{2} + (\dim \mathcal{P}_{2}-1)^{n} $ by Proposition 2.24 in \cite{Liu2016Exchange}.

    In this section, we design an evaluation algorithm and investigate the complete consistency conditions. When $ \mathcal{P}_{2,+} $ is a fusion bialgebra, we obtain algebraic constraints on parameters of exchange relation planar algebras, the so-called structure equations.
    In the notation of skein diagram, an $ n $-box is a planar graph having $ 2n $ 1-valent vertices as boundary points and some 4-valent vertices as internal crossings labelled by $ L_2 $. A $ k $-face is a face having $ k $-vertices. It should be pointed out that faces could be shaded or unshaded.

    \begin{lemma}\label{lem1}
        Suppose $ D $ is a non-empty connected 0-box in $ \mathcal{P}(L_2) $ and $ D $ has no $ 0,1,2 $-faces, then $ D $ has at least eight $ 3 $-faces. 
    \end{lemma}
    \begin{proof}
        $ D $ can be viewed as a connected $ 4 $-regular planar graph. Let $ v $ be the number of vertices, $ e $ the number of edges, $ f $ the number of faces, and $ k $ the number of $ 3 $-faces. Since each vertex has degree 4, we have $ e = 2v $. Euler's formula $ v-e+f=2 $ gives $ f=v+2 $. Since $ D $ has no $ 0,1,2 $-faces, then $ e - \frac{3k}{2}  \ge 2(f - k) $. It follows that $ k \ge 8 $.
    \end{proof} 

    \begin{lemma}\label{lem2}
        If there exists an internal face in a connected $ n $-box $ D $ for $ n = 1, 2, 3 $, then there must be an internal $ k $-face in $ D $ for $ k = 1, 2, 3 $.
    \end{lemma}
    \begin{proof}
        An $ n $-box has $ 2n $ 1-valent vertices, so counting degrees of vertices twice gives $ 2e = 4(v-2n) + 2n $. Hence, $ e = 2v-3n $. Euler's formula $ v-e+f=2 $ gives $ f = v-3n+2 $. Note that the external face contributes $ 2n $ to the number of edges via 1-valent vertices. Since there exists internal faces, $ f \ge 2 $. Thus the number of 4-valent vertices is at least $ n $, so the external face also contributes $ \frac{n}{2} $ to common edges shared with internal faces due to the connectedness of $ D $. If $ D $ has no internal 1-faces or 2-faces, each internal face contributes at least 2 to the number of edges except 3-faces. Let $ k $ be the number of $ 3 $-faces. Counting the number of edges gives $ e - 2n - \frac{3k}{2}  \ge 2(f - 1 - k) + \frac{n}{2} $, which results in $ k \ge 4 - n \ge 1 $ for $ n = 1, 2, 3 $. Hence, there must be an internal 3-face in $ D $. 
    \end{proof}

    \begin{definition}
        Let $ D $ be a non-empty connected 0-box in $ \mathcal{P}(L_2) $. We define 6 moves to evaluate $ D $. 
        \begin{enumerate}[label={\textbf{Move} $\mathbf{\arabic*}$:},leftmargin=12ex]
            \setcounter{enumi}{-1}
            \item Evaluate a $ 0 $-face using Relation \ref{circle}.
            \item Evaluate a $ 1 $-faces using Relation \ref{exr1}.
            \item Evaluate a $ 2 $-faces using Multiplication Relation \refeq{multiplication} or Convolution Relation \refeq{convolution}.
            \item Evaluate a $ 3 $-face by using Exchange Relation \refeq{exr2} to reduce it to a $ 2 $-face and evaluate the remaining $ 2 $-face immediately. 
            \item Apply Exchange Relation \ref{exr2} on an $ E_1 $-vector to transform a tangle to another. 
            \item Apply the resolution of identity to transform \begin{tikzpicture}[inner sep=0pt,minimum size=6mm,baseline={([yshift=-0.5ex]current bounding box.center)}]
                \draw[] (0,0) -- (0,0.5);
                \draw[] (0.3,0) -- (0.3,0.5);
            \end{tikzpicture}
            to
            \begin{tikzpicture}[inner sep=0pt,minimum size=6mm,baseline={([yshift=-0.5ex]current bounding box.center)}]
                \draw[] (0,0) -- (0.3,0.5);
                \draw[] (0.3,0) -- (0,0.5);
            \end{tikzpicture}.
        \end{enumerate}
        An evaluation path is a concatenation of finitely many moves. A canonical evaluation path is an evaluation path that does not contain Move--4 or Move--5.
    \end{definition}

    There are 3 ways to reduce a 3-face using Move-3. We denote them by Move-$3.a$, Move-$3.b$, and Move-$3.c$ respectively.

    \begin{theorem}\label{thm:caeva}
        A connected 0-box $ D $ can be evaluated to a scalar $ Z(D) $ by a canonical evaluation path.
    \end{theorem}
    \begin{proof}
        If $ D $ is the empty tangle, then $ Z(D) $ is 1.
        If $ D $ has a $ k $-face for $ k\le 2 $, then we can reduce the $ k $-face using Move--$ k $. If this is not the case, then $ D $ has a $ 3 $-face by Lemma \ref{lem1} which can be reduced by certain Move--3. This process reduces a face of $ D $ in each move. Thus $ D $ will be evaluated to a scalar multiple of the empty tangle after finite number of moves. Take $ Z(D) $ to be this scalar and the concatenation of these moves forms a canonical evaluation path. 
    \end{proof}

    We list the sets $ S_n $ of $ n $-boxes ($ 1 \le n \le 3 $) that do not have internal faces.

    \[S_1 =  \left\{\vcenter{\hbox{
        \begin{tikzpicture}
            \draw[] (0,0) -- (0,0.5);
        \end{tikzpicture}
        }}
        \right\}.\]

    \[S_2 = \left\{\vcenter{\hbox{
        \begin{tikzpicture}
        \draw[] (0,0) -- (0,0.5);
        \draw[] (0.3,0) -- (0.3,0.5);
    \end{tikzpicture}, 
    \begin{tikzpicture}
        \draw (0,0.5) .. controls (0,0.25) and (0.3,0.25) .. (0.3,0.5);
        \draw (0,0) .. controls (0,0.25) and (0.3,0.25) .. (0.3,0);
    \end{tikzpicture}, \begin{tikzpicture}
        \draw[] (0,0) -- (0.3,0.5);
        \draw[] (0.3,0) -- (0,0.5);
    \end{tikzpicture}
    }}
    \right\}.\]

    \[S_3 = \left\{\vcenter{\hbox{
        \begin{tikzpicture}
        \draw[] (0,0) -- (0,0.5);
        \draw[] (0.3,0) -- (0.3,0.5);
        \draw[] (0.6,0) -- (0.6,0.5);
    \end{tikzpicture},
    \begin{tikzpicture}
        \draw (0,0.5) .. controls (0,0.25) and (0.3,0.25) .. (0.3,0.5);
        \draw (0,0) .. controls (0,0.25) and (0.3,0.25) .. (0.3,0);
        \draw[] (0.5,0) -- (0.5,0.5);
    \end{tikzpicture},
    \begin{tikzpicture}
        \draw (0,0.5) .. controls (0,0.25) and (0.3,0.25) .. (0.3,0.5);
        \draw (0,0) .. controls (0,0.25) and (0.3,0.25) .. (0.3,0);
        \draw[] (-0.2,0) -- (-0.2,0.5);
    \end{tikzpicture},
    \begin{tikzpicture}
        \draw (0,0.5) .. controls (0,0.25) and (0.3,0.25) .. (0.3,0.5);
        \draw (0.3,0) .. controls (0.3,0.25) and (0.6,0.25) .. (0.6,0);
        \draw[] (0,0) -- (0.6,0.5);
    \end{tikzpicture}, 
    \begin{tikzpicture}
        \draw (0.3,0.5) .. controls (0.3,0.25) and (0.6,0.25) .. (0.6,0.5);
        \draw (0,0) .. controls (0,0.25) and (0.3,0.25) .. (0.3,0);
        \draw[] (0,0.5) -- (0.6,0);
    \end{tikzpicture},
    \begin{tikzpicture}
        \draw[] (0,0) -- (0.3,0.5);
        \draw[] (0.3,0) -- (0,0.5);
        \draw[] (0.5,0) -- (0.5,0.5);
    \end{tikzpicture},
    \begin{tikzpicture}
        \draw[] (0,0) -- (0.3,0.5);
        \draw[] (0.3,0) -- (0,0.5);
        \draw[] (-0.2,0) -- (-0.2,0.5);
    \end{tikzpicture},
    \begin{tikzpicture}
        \draw (0,0.5) .. controls (0,0.25) and (0.6,0.25) .. (0.6,0.5);
        \draw (0,0) .. controls (0,0.25) and (0.3,0.25) .. (0.3,0);
        \draw[] (0.3,0.5) -- (0.6,0);
    \end{tikzpicture},
    \begin{tikzpicture}
        \draw (0,0.5) .. controls (0,0.25) and (0.3,0.25) .. (0.3,0.5);
        \draw (0,0) .. controls (0,0.25) and (0.6,0.25) .. (0.6,0);
        \draw[] (0.3,0) -- (0.6,0.5);
    \end{tikzpicture},
    \begin{tikzpicture}
        \draw (0,0.5) .. controls (0,0.25) and (0.6,0.25) .. (0.6,0.5);
        \draw (0.3,0) .. controls (0.3,0.25) and (0.6,0.25) .. (0.6,0);
        \draw[] (0.3,0.5) -- (0,0);
    \end{tikzpicture},
    \begin{tikzpicture}
        \draw (0.3,0.5) .. controls (0.3,0.25) and (0.6,0.25) .. (0.6,0.5);
        \draw (0,0) .. controls (0,0.25) and (0.6,0.25) .. (0.6,0);
        \draw[] (0.3,0) -- (0,0.5);
    \end{tikzpicture},
    \begin{tikzpicture}
        \draw[] (0,0) -- (0.3,0.5);
        \draw[] (0.3,0) -- (0.6,0.5);
        \draw[] (0,0.5) -- (0.6,0);
    \end{tikzpicture},
    \begin{tikzpicture}
        \draw (0,0.5) .. controls (0,0.25) and (0.6,0.25) .. (0.6,0.5);
        \draw (0,0) .. controls (0,0.25) and (0.6,0.25) .. (0.6,0);
        \draw[] (0.3,0) -- (0.3,0.5);
    \end{tikzpicture},
    \begin{tikzpicture}
        \draw[] (0,0.5) -- (0.3,0);
        \draw[] (0.3,0.5) -- (0.6,0);
        \draw[] (0,0) -- (0.6,0.5);
    \end{tikzpicture}
    }}
    \right\}.\]

    \begin{theorem}\label{thm:evathm}
        Different canonical evaluation paths on any connected 0-box give the same scalar if different canonical evaluation paths on the tangles which are the inner product of the following $ n $-boxes and elements of $ S_n $ for $ n = 1, 2, 3 $ give the same scalar. 
        \begin{figure}[ht]
            \centering
            \subfloat[][]{
                \begin{tikzpicture}
                \draw (0,0)  .. controls (1,1) and (1,-1) ..  (0,0);
                \draw (0,0)  .. controls (-1,-1) and (-1,1) ..  (0,0);
                \end{tikzpicture}
                \label{eva_a}
            }
            \quad
            \subfloat[][]{
                \begin{tikzpicture}
                    \draw (0,0)  .. controls (1,1) and (1,-1) ..  (0,0);
                    \draw (0,0) to[bend left=45] (-1,0)--(-1.2,0.2) ;
                    \draw (0,0) to[bend right=45] (-1,0)--(-1.2,-0.2);
                    \draw (0,0) to[out=45, in=-45, looseness=4] (0,0);
                \end{tikzpicture}
                \label{eva_b}
            }
            \quad
            \subfloat[][]{
                \begin{tikzpicture}
                    \draw (-1,0) -- (1,0);
                    \draw (0,0) circle (0.5);
                \end{tikzpicture}
                \label{eva_c}
            }
            \quad
            \subfloat[][]{
                \begin{tikzpicture}
                    \draw (0,0) to[bend left=45] (-1,0)--(-1.2,0.2) ;
                    \draw (0,0) to[bend right=45] (-1,0)--(-1.2,-0.2);
                    \draw (0,0) to[bend left=45] (1,0)--(1.2,0.2) ;
                    \draw (0,0) to[bend right=45] (1,0)--(1.2,-0.2);
                \end{tikzpicture}
                \label{eva_d}
            }
            \quad
            \subfloat[][]{
                \begin{tikzpicture}
                    \draw (0,1) -- (1.2,0);
                    \draw (0.3,1) -- (0.3,0);
                    \draw (1.2,1) -- (0,0);
                \end{tikzpicture}
                \label{eva_e}
            }
            \quad 
            \subfloat[][]{
                \begin{tikzpicture}
                    \draw (0,1) -- (1.5,0);
                    \draw (0.5,0) -- (0.5,0.9) ..controls (0.5,1.1) and (1.2,0.8)..(1,0.666) -- (0,0);
                \end{tikzpicture}
                \label{eva_f}
            }

            \caption{Evaluation tangles.}
            \label{fig:cont}
        \end{figure}
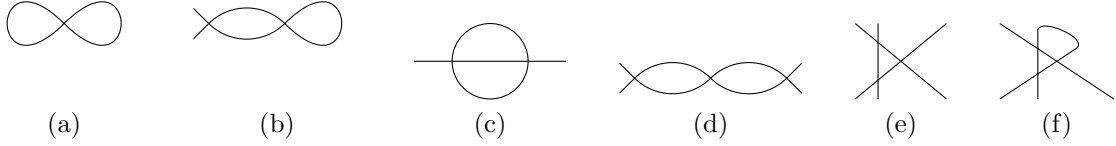
    \end{theorem}
    \begin{proof}
        Let $ D $ be a connected 0-box.
        We aim to prove that the scalar $ Z(D) $ is independent of the canonical evaluation path, specifically, the starting move on starting face, by induction on the number of faces of $ D $. 

        For the base case, both the empty tangle and the 0-face tangle have a unique canonical evaluation path. We now consider tangles with at least 2 faces. It is important to note that any two moves on non-adjacent parts of a tangle are commutative. Thus the order in which these moves are applied does not affect the resulting tangle.

        Suppose $ D_n $ and $ D_k $ ($ 0\le n, k\le 3 $) are obtained by reducing an $ n $-face $ F_n $ and a $ k $-face $ F_k $ in $ D $ respectively, and that $ F_n $ and $ F_k $ are not adjacent. By the induction hypothesis, $ Z(D_n) $ is independent of the starting face, so we can evaluate $ F_k $ in $ D_n $ as the starting move to obtain a tangle $ D_{n,k} $. Similarly, we can evaluate $ F_n $ in $ D_k $ as the first move and obtain $ D_{n,k} $ as well. By concatenating a canonical evaluation path of $ D_{n,k} $ with above two evaluation paths respectively, we conclude that choosing either Move--$ n $ or Move--$ k $ as the starting move will give the same scalar for $ D $.

        Next, we prove that 3 distinct Move--3 are the same transform under our assumption.
        The complement of $ F_3 $ in $ D $ is also a 3-box $ B_3 $. If $ B_3 $ has no internal faces, then this case reduces to our assumption. If $ B_3 $ has internal faces, then there must be an internal $ k $-face $ F_k $ for $ k \le 3 $ by Lemma \ref{lem2}.  Note that Move-$3.a$ or Move-$ 3.b $ ($ a,b = 1, 2, 3 $) on $ F_3 $ and Move-$ k $ on $ F_k $ are commutative since they are not adjacent, so Move-$3.a$ or Move-$ 3.b $ on $ F_3 $ are commutative as starting move of $ D $, which means choosing whichever one as the starting move to evaluate $ F_3 $ will give the same scalar for $ D $.

        Now we consider the case that $ F_n $ and $ F_k $ are adjacent and $ n \ge k $. Then $ k $ cannot be 0 since a $ 0 $-face is closed. If $ n\le 2 $, then possible pairs $ (n,k) $ can only be $ (1,1), (2,1)$, and $ (2,2) $. The adjacency of $ F_n $ and $ F_k $ is illustrated in Figure \ref{eva_a}, \ref{eva_b},\ref{eva_c}, \ref{eva_d}. Figure \ref{eva_a} reduces to our assumption. For Figure \ref{eva_b}, \ref{eva_c}, \ref{eva_d}, $ D $ can be viewed as the inner product of the adjacency part $ A $ and the complement part $ B $. If $ B $ has no internal faces, then this case also reduces to our assumption, too. If $ B $ has internal faces, then by Lemma \ref{lem2}, there must be a 1-face, or 2-face, or 3-face, which we denote as $ F_c $. According to above argument, Choosing $ F_n $ or $ F_c $ as the starting face gives the same scalar. A similar result holds for $ F_k $ and $ F_c $. Therefore, choosing either $ F_n $ or $ F_k $ as the starting face will give the same scalar.

        For $ n = 3 $, the possible pairs $ (n,k) $ can be $ (3,1), (3,2), (3,3) $. Different adjacency configurations are shown in Figure \ref{eva_f} and Figure \ref{eva_g}, \ref{eva2_a}, \ref{eva2_b}, \ref{eva2_c}. The same argument applies to Figure \ref{eva_f}.

        \begin{figure}[ht]
            \centering

            \subfloat[][]{
                \begin{tikzpicture}
                    \draw (0,1) -- (1.2,0);
                    \draw (0.3,1) -- (0.3,0);
                    \draw (1.2,1) -- (0,0);
                    \draw[] (0,1) to[out=140, in=90, looseness = 4] (0.3,1);
                \end{tikzpicture}
                \label{eva_g}
            }
            \quad
            \subfloat[][]{
                \begin{tikzpicture}
                    \draw (1.5,1) -- (0,0);
                    \draw (0.5,0) -- (0.5,0.8) .. controls (0.5,1.2).. (0,1.3);
                    \draw (1.5,0) -- (0.3,0.8) .. controls (0.2,0.866) .. (0.15,1.4);
                \end{tikzpicture}
                \label{eva2_a}
            }
            \quad
            \subfloat[][]{
                \begin{tikzpicture}
                    \draw (0,1) -- (1.5,0);
                    \draw (0.5,1) -- (0.5,0);
                    \draw (1,1) -- (1,0);
                    \draw (1.5,1) -- (0,0);
                \end{tikzpicture}
                \label{eva2_b}
            }
            \quad
            \subfloat[][]{
                \begin{tikzpicture}
                    \draw (0,0) .. controls (0,1) and (2,1) .. (2,0);
                    \draw (0,1) .. controls (0,0) and (2,0) .. (2,1);
                    \draw (1,0) -- (1,1);
                    \end{tikzpicture}
                    \label{eva2_c}
            }
            \caption{}
            \label{<label>}
        \end{figure}

        We proceed to prove that the commutativity of moves on Figure \ref{eva_g}, \ref{eva2_a}, \ref{eva2_b}, \ref{eva2_c} can be deduced. Now we may use Move--3 arbitrarily. Note that Move--3 is the concatenation of Move--4 and Move--2. In Figure \ref{eva2_a}, we choose two crossings in $ F_3 $ that are not in $ F_2 $ to apply Move--4, resulting in the configuration \begin{tikzpicture}[inner sep=0pt,minimum size=6mm,baseline={([yshift=-0.5ex]current bounding box.center)}]
            \draw (0,0) to[bend left=45] (-1,0)--(-1.2,0.2) ;
            \draw (0,0) to[bend right=45] (-1,0)--(-1.2,-0.2);
            \draw (0,0) to[bend left=45] (1,0)--(1.2,0.2) ;
            \draw (0,0) to[bend right=45] (1,0)--(1.2,-0.2);
            \draw (-1,-0.2) -- (-1.2,0);
        \end{tikzpicture} and \begin{tikzpicture}[inner sep=0pt,minimum size=6mm,baseline={([yshift=-0.5ex]current bounding box.center)}]
            \draw (0,0) to[bend left=45] (-1,0)--(-1.2,0.2) ;
            \draw (0,0) to[bend right=45] (-1,0)--(-1.2,-0.2);
            \draw (0,0) to[bend left=45] (1,0)--(1.2,0.2) ;
            \draw (0,0) to[bend right=45] (1,0)--(1.2,-0.2);
            \draw (-1,0.2) -- (-1.2,0);
        \end{tikzpicture}. We then apply Move--2 immediately to reduce $ F_2^{'} $ to complete Move--3, and we apply Move--2 to reduce $ F_2 $ by induction. By what we have proved on Figure \ref{eva_d}, we can reduce $ F_2 $ first and then $ F_2^{'} $. Furthermore, since Move--2 on $ F_2 $ and Move--4 on $ F_3 $ act on non-adjacent parts, they are commutative. Thus evaluating either $ F_3 $ or $ F_2 $ as the starting move gives the same scalar.  In Figure \ref{eva2_b}, we choose two pairs of non-adjacent crossings to apply exchange relations and obtain \begin{tikzpicture}[inner sep=0pt,minimum size=6mm,baseline={([yshift=-0.5ex]current bounding box.center)}]
            \draw (0,0) to[bend left=45] (-1,0)--(-1.2,0.2) ;
            \draw (0,0) to[bend right=45] (-1,0)--(-1.2,-0.2);
            \draw (0,0) to[bend left=45] (1,0)--(1.2,0.2) ;
            \draw (0,0) to[bend right=45] (1,0)--(1.2,-0.2);
            \draw (1,-0.2) -- (1.2,0);
            \draw (-1,-0.2) -- (-1.2,0);
        \end{tikzpicture}, \begin{tikzpicture}[inner sep=0pt,minimum size=6mm,baseline={([yshift=-0.5ex]current bounding box.center)}]
            \draw (0,0) to[bend left=45] (-1,0)--(-1.2,0.2) ;
            \draw (0,0) to[bend right=45] (-1,0)--(-1.2,-0.2);
            \draw (0,0) to[bend left=45] (1,0)--(1.2,0.2) ;
            \draw (0,0) to[bend right=45] (1,0)--(1.2,-0.2);
            \draw (1,0.2) -- (1.2,0);
            \draw (-1,-0.2) -- (-1.2,0);
        \end{tikzpicture}, \begin{tikzpicture}[inner sep=0pt,minimum size=6mm,baseline={([yshift=-0.5ex]current bounding box.center)}]
            \draw (0,0) to[bend left=45] (-1,0)--(-1.2,0.2) ;
            \draw (0,0) to[bend right=45] (-1,0)--(-1.2,-0.2);
            \draw (0,0) to[bend left=45] (1,0)--(1.2,0.2) ;
            \draw (0,0) to[bend right=45] (1,0)--(1.2,-0.2);
            \draw (1,-0.2) -- (1.2,0);
            \draw (-1,0.2) -- (-1.2,0);
        \end{tikzpicture}, and \begin{tikzpicture}[inner sep=0pt,minimum size=6mm,baseline={([yshift=-0.5ex]current bounding box.center)}]
            \draw (0,0) to[bend left=45] (-1,0)--(-1.2,0.2) ;
            \draw (0,0) to[bend right=45] (-1,0)--(-1.2,-0.2);
            \draw (0,0) to[bend left=45] (1,0)--(1.2,0.2) ;
            \draw (0,0) to[bend right=45] (1,0)--(1.2,-0.2);
            \draw (1,0.2) -- (1.2,0);
            \draw (-1,0.2) -- (-1.2,0);
        \end{tikzpicture}. Thus the consistency of evaluating $ F_3 $ or $ F_3^{'} $ first follows from the conclusion on Figure \ref{eva_d}. In Figure \ref{eva2_c}, we choose two common crossings of two $ 3 $-faces $ F_3 $ and $ F_3^{'} $ to apply exchange relation and obtain \begin{tikzpicture}[scale=0.5]
            \draw (0,0) .. controls (0,1) and (2,1) .. (2,0);
            \draw (3,1) .. controls (3,0) and (5,0) .. (5,1);
            \draw (0,1) .. controls (0,0) and (1.5,0) .. (2.5,0.5) ;
            \draw (5,0) .. controls (5,1) and (3.5,1) .. (2.5,0.5);
            \end{tikzpicture} and \begin{tikzpicture}[scale=0.5]
                \draw (0,1) .. controls (0,0) and (2,0) .. (2,1);
                \draw (3,0) .. controls (3,1) and (5,1) .. (5,0);
                \draw (0,0) .. controls (0,1) and (1.5,1) .. (2.5,0.5) ;
                \draw (5,1) .. controls (5,0) and (3.5,0) .. (2.5,0.5);
            \end{tikzpicture}. Evaluating $ F_3 $ first or $ F_3^{'} $ first ultimately depends on which $ 2 $-face to evaluate first. Since the two $ 2 $-faces are not adjacent, evaluating them in different orders gives the same scalar.
    \end{proof}

    \begin{definition}
        The partition function $ Z $ on a 0-box $ D $ is defined as the product of the values obtained from canonical evaluation paths for connected components of $ D $.
    \end{definition}

    $ Z $ is invariant under Move--$ k $ for $ k = 0,1,2,3 $ since each move can be regarded as a starting move of a canonical evaluation path. This is a direct extension from a connected 0-box to a 0-box having multiple connected components. However, the resolution of identity may join two connected components.

    \begin{theorem}\label{thm:conthm}
        $ Z $ is invariant on any 0-box under Move--4 and Move--5 if $ Z $ is invariant on the inner product of 
        \begin{enumerate}[label=(\arabic*)]
            \item \begin{tikzpicture}[baseline=0.65ex]
                \draw[] (0,0) -- (0.3,0.5);
                \draw[] (0.3,0) -- (0.6,0.5);
                \draw[] (0,0.5) -- (0.6,0);
            \end{tikzpicture} and elements of $ S_3 $,
            \item \begin{tikzpicture}[baseline=0.65ex]
                \draw[] (0,0) -- (0,0.5);
                \draw[] (0.3,0) -- (0.3,0.5);
            \end{tikzpicture} and elements of $ S_2 $,
        \end{enumerate}
        under Move--4 and Move--5.
    \end{theorem}
    \begin{proof}
        Let $ D $ be a 0-box. Let $ \widetilde{D} $ be obtained by applying Move--4 or Move--5 in $ D $. It suffices to show $ Z(D) = Z(\widetilde{D}) $. We prove by induction on the number of faces of $ D $. 

        Move--4 transforms an $ E_1 $-vector to a linear sum of $ E_2,E_3 $-vectors. The complement of the $ E_1 $-vector is a 3-box $ B_3 $. If $ B_3 $ has no internal faces, then this case reduces to assumption (1). 
        If there exists an internal $ k $-face $ F_k $ ($ k = 1, 2, 3 $) in $ B_3 $, then we can evaluate $ D $ starting from $ F_k $ to obtain a tangle $ D_k $. Hence, $ Z(D) = Z(D_k) $ according to Theorem \ref{thm:evathm}. Similarly, we can select $ F_k $ as the starting face to evaluate  $ \widetilde{D} $, obtaining a tangle $ \widetilde{D}_k $, and $ Z(\widetilde{D}) = Z(\widetilde{D}_k) $. By induction, we can apply Move--4 on the $ E_1 $-vector in $ D_k $ to obtain another tangle without changing the value of the partition function. Note that this tangle is exactly $ \widetilde{D}_k $. Therefore, $ Z(D) = Z(D_k) = Z(\widetilde{D}_k) = Z(\widetilde{D}) $.

        The complement of \begin{tikzpicture}[baseline=0.65ex]
            \draw[] (0,0) -- (0,0.5);
            \draw[] (0.3,0) -- (0.3,0.5);
        \end{tikzpicture} is a 2-box $ B_2 $. If they form a connected tangle, then the argument is similar to above by induction hypothesis. If they form two connected tangles, then $ B_2 $ must be disconnected. We prove by induction on the pair of the number of faces of two connected components of $ B_2 $. If $ B_2 = \vcenter{\hbox{\begin{tikzpicture}
            \draw[] (0,0) -- (0,0.5);
            \draw[] (0.3,0) -- (0.3,0.5);
        \end{tikzpicture}}} $, then this case reduces to assumption (2). Otherwise, there exists an internal $ k $-face $ F_k $ $ (k = 1, 2, 3) $ in one connected component of $ B_2 $. The similar argument applies.
    \end{proof}

    \begin{remark}
        Theorem \ref{thm:caeva}, \ref{thm:evathm}, \ref{thm:conthm} establishes that the consistency conditions of exchange relations can be reduced to the consistency of finitely many evaluation paths on finitely many 0-boxes. Hence, the consistency conditions are equivalent to finitely many algebraic equations on the parameters. However, it is not practical to solve these equations directly in the generic case. 
    \end{remark}

    When restricted to the fusion bialgebra case, we have a canonical basis of $ \mathcal{P}_{2,+} $ and $ \mathcal{P}_{3,+} $, which simplifies the equations. We can compare coefficients on the basis instead of computing inner product.

    \begin{theorem}\label{thm:polyeq}
        Suppose $ \mathcal{B} $ is a subfactorizable fusion bialgebra with exchange relations.
        Let $ \{N_{ij}^{k}\}, \{d_i\}, \{a_{st}^{ij}\}, \{b_{st}^{ij}\} (0 \le i, j, k, s, t \le n) $ be defined as in Formula \ref{convolution} and Def. \ref{Def: Exchange Relation}. Then consistency conditions of exchange relations are equivalent to basic consistency equations in Proposition \ref{basic_cons} and the following system of polynomial equations. 
        \begin{gather}
            d_i = d_{\overline{i}}, \label{trace} \\
            d_id_j = \sum_{s=0}^{n} N_{ij}^{s} d_s, \label{coptrace} \\
            \sum_{s=0}^{n} N_{ij}^{s}N_{sk}^{\ell} = \sum_{s=0}^{n} N_{is}^{\ell}N_{jk}^{s},\label{associativity} \\
            N_{kl}^{s}  (a_{sk}^{ij} + b_{lk}^{ij} - \delta_{si}\delta_{\ell j}) = 0, \label{3face_1} \\
            N_{zy}^{x}\sum_{s=0}^{n} (a_{sz}^{jk}N_{si}^{x} + b_{sz}^{jk}N_{si}^{y} - a_{s\overline{x}}^{\overline{i}j}N_{s\overline{k}}^{\overline{y}} - b_{s\overline{x}}^{\overline{i}j}N_{s\overline{k}}^{z}) = 0, \label{3face_2} \\
            N_{zy}^{x}\sum_{s=0}^{n} (a_{sz}^{jk}N_{si}^{x} + b_{sz}^{jk}N_{si}^{y} - a_{s\overline{x}}^{\overline{i}j}N_{s\overline{k}}^{\overline{y}} - b_{s\overline{x}}^{\overline{i}j}N_{s\overline{k}}^{z}) = 0, \label{3face_3} \\
            N_{i\overline{j}}^{k} = \sum_{s=0}^{n} (a_{sk}^{ij}+b_{sk}^{ij})d_s, \label{E_2face_1}
        \end{gather}
    \end{theorem}
    \begin{proof}
        We consider all tangles in the consistency conditions from Theorem \ref{thm:evathm}, \ref{thm:conthm}. They have two kinds of shadings.

        Figure \ref{eva_a} shows two different ways to compute trace, and this has been proven in Proposition \ref{basic_cons} to be equivalent to Equation \ref{trace}. 

        There is only one non-trivial shading on Figure \ref{eva_b}. On one hand, we can evaluate convolution and then evaluate 1-face. On the other hand, we can evaluate 1-faces twice. Hence, we obtain Equation \ref{coptrace}.

        Figure \ref{eva_c} shows evaluation of two 2-faces using the same crossings. This gives $ \delta_{ij}d_i = N_{i\overline{j}}^{0} $, which is a basic consistency equation in Proposition \ref{basic_cons}.

        Figure \ref{eva_d} shows the associativity of multiplication and convolution. Hence, we obtain Equation \ref{associativity}. Take $ \ell = 0 $ in Equation \ref{associativity}, we obtain Frobenius reciprocity.

        Figure \ref{eva_e} is a $ 3 $-face $ T(i,j,k) $ or $ W(i,j,k) $ depending on the shading. We first evaluate $ T(i,j,k) $ in 3 ways.
        Using exchange relation \ref{exr2} on $ p_i $ and $ p_j $ gives 
        \begin{equation}
            \begin{aligned}
                T(i,j,k) &= \sum_{s=0}^{n} a_{sk}^{ij}E_2(s,k) + \sum_{\ell = 0}^{n} b_{\ell k}^{ij} E_3(\ell,k) \\
                &= \sum_{s,\ell = 0}^{n} (a_{sk}^{ij}+b_{\ell k}^{ij})T(s,l,k).
            \end{aligned}
        \end{equation}
        
        Note that $ T(s,\ell,k) = 0 \Longleftrightarrow N_{k\ell}^{s} = 0 $. If $ N_{kj}^{i} \ne 0 $, we have $ a_{ik}^{ij}+b_{jk}^{ij} = 1 $. If $ N_{k\ell}^{s} \ne 0 $ and $ s \ne i $ or $ \ell \ne j $, we have $ a_{sk}^{ij}+b_{\ell k}^{ij} = 0 $. We combine the two situations into Equation \ref{3face_1}. Similarly, using exchange relation \ref{exr2} on $ p_j $ and $ p_k $, $ p_k $ and $ p_i $ gives 
        \begin{align}
            N_{\ell \overline{s}}^{i}(a_{s\overline{i}}^{\overline{j}k}+b_{\ell \overline{i}}^{\overline{j}k} - \delta_{s\overline{j}}\delta_{\ell k}) = 0, \\
            N_{\overline{s}j}^{\overline{\ell}}(a_{sj}^{\overline{k}\overline{i}}+b_{\ell j}^{\overline{k}\overline{i}} - \delta_{s\overline{k}}\delta_{\ell \overline{i}}) = 0.
        \end{align}
        Frobenius reciprocity requires $ N_{\ell \overline{s}}^{i} \ne 0 \Leftrightarrow N_{\overline{i} \ell}^{s} \ne 0 $, so the above two equations are equivalent to Equation \ref{3face_1}. 

        We now evaluate $ W(i,j,k) $ in 3 different ways. Choosing any two of $ p_i, p_j $ and $ p_k $ to use exchange relation \ref{exr2} gives 
        \begin{align}
            W(i,j,k) & = \sum_{s,z = 0}^{n} a_{sz}^{jk}\sum_{x=0}^{n} N_{si}^{x}E_2(x,z) + \sum_{s,z = 0}^{n} b_{sz}^{jk}\sum_{y=0}^{n} N_{si}^{y}E_3(y,z) \\
            & = \sum_{z,x,y=0}^{n} \left[\sum_{s=0}^{n} (a_{sz}^{jk}N_{si}^{x} + b_{sz}^{jk}N_{si}^{y})\right]T(x,y,z), \\
            W(i,j,k) & = \sum_{s,x = 0}^{n} a_{s\overline{x}}^{\overline{i}j}\sum_{y=0}^{n} N_{s\overline{k}}^{\overline{y}}E_1(x,y) + \sum_{s,x = 0}^{n} b_{s\overline{x}}^{\overline{i}j}\sum_{z=0}^{n} N_{s\overline{k}}^{z}E_2(x,z) \\
            & = \sum_{z,x,y=0}^{n} \left[\sum_{s=0}^{n} (a_{s\overline{x}}^{\overline{i}j}N_{s\overline{k}}^{\overline{y}} + b_{s\overline{x}}^{\overline{i}j}N_{s\overline{k}}^{z})\right]T(x,y,z), \\  
            W(i,j,k) & = \sum_{s,y = 0}^{n} a_{sy}^{k\overline{i}}\sum_{z=0}^{n} N_{s\overline{j}}^{\overline{z}}E_3(y,z) + \sum_{s,y = 0}^{n} b_{sy}^{k\overline{i}}\sum_{x=0}^{n} N_{s\overline{j}}^{\overline{x}}E_1(x,y) \\
            & = \sum_{z,x,y=0}^{n} \left[\sum_{s=0}^{n} (a_{sy}^{k\overline{i}}N_{s\overline{j}}^{\overline{z}} + b_{sy}^{k\overline{i}}N_{s\overline{j}}^{\overline{x}})\right]T(x,y,z). 
        \end{align}
        Compare coefficients of $ T(x,y,z) $, and we obtain Equation \ref{3face_2} and \ref{3face_3}.

        One shading on Figure \ref{eva_f} gives 
        \begin{align}
            \begin{tikzpicture}[baseline={([yshift=-0.5ex]current bounding box.center)}]
                \node (i) at (0,1) [rectangle, draw]{$ p_i $};
                \node (j) at (1,0) [rectangle, draw, rotate=90] {$ p_j $};
                \node (k) at (0,-1) [rectangle, draw]{$ p_k $};            
                \draw (0.15,0.7) .. controls +(down:5mm) and+(left:5mm) .. (0.7,0.15);
                \draw (0.15,-0.7) .. controls +(up:5mm) and+(left:5mm) .. (0.7,-0.15);
                \draw (1.3,0.15) .. controls +(right:5mm) and +(down:5mm) .. (1.8,0.7) -- (1.8,1.3) ;
                \draw (1.3,-0.15) .. controls +(right:5mm) and +(up:5mm) .. (1.8,-0.7) -- (1.8,-1.6) ;
                \draw (-0.15,0.7) -- (-0.15,-0.7);
                \draw (-0.15,1.3) -- (-0.15,1.6);
                \draw (0.15,-1.3) -- (0.15,-1.6);
                \draw (-0.15,-1.3) -- (-0.15,-1.6);
                \draw (0.15,1.3) .. controls +(up:3mm) and +(left:3mm) .. (0.45,1.6) -- (1.5,1.6);
                \draw (1.5,1.6) .. controls +(right:3mm) and +(up:3mm) .. (1.8,1.3);
            \end{tikzpicture}
            &= N_{i\overline{j}}^{k} p_k, \\
            \begin{tikzpicture}[baseline={([yshift=-0.5ex]current bounding box.center)}]
                \node (i) at (0,1) [rectangle, draw]{$ p_i $};
                \node (j) at (1,0) [rectangle, draw, rotate=90] {$ p_j $};
                \node (k) at (0,-1) [rectangle, draw]{$ p_k $};            
                \draw (0.15,0.7) .. controls +(down:5mm) and+(left:5mm) .. (0.7,0.15);
                \draw (0.15,-0.7) .. controls +(up:5mm) and+(left:5mm) .. (0.7,-0.15);
                \draw (1.3,0.15) .. controls +(right:5mm) and +(down:5mm) .. (1.8,0.7) -- (1.8,1.3) ;
                \draw (1.3,-0.15) .. controls +(right:5mm) and +(up:5mm) .. (1.8,-0.7) -- (1.8,-1.6) ;
                \draw (-0.15,0.7) -- (-0.15,-0.7);
                \draw (-0.15,1.3) -- (-0.15,1.6);
                \draw (0.15,-1.3) -- (0.15,-1.6);
                \draw (-0.15,-1.3) -- (-0.15,-1.6);
                \draw (0.15,1.3) .. controls +(up:3mm) and +(left:3mm) .. (0.45,1.6) -- (1.5,1.6);
                \draw (1.5,1.6) .. controls +(right:3mm) and +(up:3mm) .. (1.8,1.3);
            \end{tikzpicture} &= \sum_{s=0}^{n} a_{sk}^{ij}\ \begin{tikzpicture}[inner sep=0pt,minimum size=6mm,baseline={([yshift=-0.5ex]current bounding box.center)}]
                \node (s) at (0,1) [rectangle, draw]{$ p_s $};
                \node (t) at (0,0) [rectangle, draw] {$ p_k $};
                \draw (-0.15,0.7) -- (-0.15,0.3);
                \draw (0.15,1.3) ..controls (0.15,1.6) and (0.45,1.6) .. (0.45,1.3);
                \draw (-0.15,1.3) -- (-0.15,1.6);
                \draw (0.15,-0.3) -- (0.15,-0.7);
                \draw (-0.15,-0.3) -- (-0.15,-0.7);
                \draw (0.15,0.7) .. controls (0.15,0.4) and(0.45,0.4) .. (0.45,0.7) -- (0.45,1.3);
                \draw (0.15,0.3) .. controls (0.3,0.5) and(0.7,0.5) .. (0.7,-0.7);
            \end{tikzpicture} + \sum_{s=0}^{n} b_{sk}^{ij} \ 
            \begin{tikzpicture}[inner sep=0pt,minimum size=6mm,baseline={([yshift=-0.5ex]current bounding box.center)}]
                \node (t) at (0,-1) [rectangle, draw]{$ p_k $};
                \node (s) at (1,0) [rectangle, draw, rotate=90] {$ p_s $};
                \draw (0.2,0.35) .. controls +(down:2mm) and+(left:5mm) .. (0.7,0.15);
                \draw (0.2,0.35) .. controls +(up:2mm) and+(left:5mm) .. (0.7,0.7) -- (1.3,0.7);
                \draw (1.3,0.7) .. controls +(right:5mm) and +(up:2mm) .. (1.8,0.35);
                \draw (0.15,-0.7) .. controls +(up:5mm) and+(left:5mm) .. (0.7,-0.15);
                \draw (1.3,0.15) .. controls +(right:5mm) and +(down:2mm) .. (1.8,0.35);
                \draw (1.3,-0.15) .. controls +(right:5mm) and +(up:5mm) .. (1.8,-0.7) -- (1.8,-1.6) ;
                \draw (-0.15,0.7) -- (-0.15,-0.7);
                \draw (0.15,-1.3) -- (0.15,-1.6);
                \draw (-0.15,-1.3) -- (-0.15,-1.6);
            \end{tikzpicture}  \\ 
            &= \sum_{s=0}^{n} (a_{sk}^{ij} + b_{sk}^{ij})d_s p_k. 
        \end{align}
        Compare coefficients of $ p_k $, and we get Equation \ref{E_2face_1}.

        Another shading on Figure \ref{eva_f} gives
        \begin{align}
            \begin{tikzpicture}[inner sep=0pt,minimum size=6mm,baseline={([yshift=-0.5ex]current bounding box.center)}]
                \node (i) at (0,1) [rectangle, draw]{$ p_i $};
                \node (j) at (1,0) [rectangle, draw, rotate=90] {$ p_j $};
                \node (k) at (1,2) [rectangle, draw, rotate=90] {$ p_k $};
                \draw (0.15,1.3) .. controls +(up:5mm) and+(left:5mm) .. (0.7,1.85);
                \draw (0.15,0.7) .. controls +(down:5mm) and+(left:5mm) .. (0.7,0.15);
                \draw (0.15,-0.4) .. controls +(up:2mm) and+(left:5mm) .. (0.7,-0.15);
                \draw (1.3,0.15) .. controls +(right:5mm) and +(down:5mm) .. (1.8,0.7) -- (1.8,1.3);
                \draw (1.3,2.15) .. controls +(right:5mm) and +(down:2.5mm) .. (1.8,2.4);
                \draw (0.7,2.15) .. controls +(left:5mm) and +(down:2.5mm) .. (0.15,2.4);
                \draw (1.3,-0.15) .. controls +(right:5mm) and +(up:5mm) .. (1.8,-0.7) ;
                \draw (1.3,1.85) .. controls +(right:5mm) and +(up:5mm) .. (1.8,1.3) ;
                \draw (-0.15,0.7) -- (-0.15,-0.4);
                \draw (-0.15,-0.4) .. controls (-0.15,-0.7) and (0.15,-0.7) .. (0.15,-0.4);
                \draw (-0.15,1.3) -- (-0.15,2.4);
            \end{tikzpicture} &= \sum_{\ell = 0}^{n} \delta_{ij} N_{ik}^{\ell}p_{\ell}, \\
            \begin{tikzpicture}[inner sep=0pt,minimum size=6mm,baseline={([yshift=-0.5ex]current bounding box.center)}]
                \node (i) at (0,1) [rectangle, draw]{$ p_i $};
                \node (j) at (1,0) [rectangle, draw, rotate=90] {$ p_j $};
                \node (k) at (1,2) [rectangle, draw, rotate=90] {$ p_k $};
                \draw (0.15,1.3) .. controls +(up:5mm) and+(left:5mm) .. (0.7,1.85);
                \draw (0.15,0.7) .. controls +(down:5mm) and+(left:5mm) .. (0.7,0.15);
                \draw (0.15,-0.4) .. controls +(up:2mm) and+(left:5mm) .. (0.7,-0.15);
                \draw (1.3,0.15) .. controls +(right:5mm) and +(down:5mm) .. (1.8,0.7) -- (1.8,1.3);
                \draw (1.3,2.15) .. controls +(right:5mm) and +(down:2.5mm) .. (1.8,2.4);
                \draw (0.7,2.15) .. controls +(left:5mm) and +(down:2.5mm) .. (0.15,2.4);
                \draw (1.3,-0.15) .. controls +(right:5mm) and +(up:5mm) .. (1.8,-0.7) ;
                \draw (1.3,1.85) .. controls +(right:5mm) and +(up:5mm) .. (1.8,1.3) ;
                \draw (-0.15,0.7) -- (-0.15,-0.4);
                \draw (-0.15,-0.4) .. controls (-0.15,-0.7) and (0.15,-0.7) .. (0.15,-0.4);
                \draw (-0.15,1.3) -- (-0.15,2.4);
            \end{tikzpicture} &= \sum_{s,\ell=0}^{n} (a_{s0}^{ij} + b_{s0}^{ij}) N_{sk}^{\ell}p_{\ell}.
        \end{align}
        Compare coefficients of $ p_{\ell} $, we get $ \sum\limits_{s=0}^{n} (a_{s0}^{ij}+b_{s0}^{ij})N_{sk}^{\ell} = \delta_{ij}N_{ik}^{\ell} $, which can be deduced from Equation \ref{3face_1}.

        Condition (1) in Theorem \ref{thm:conthm} involves 14 tangles, only 3 of them may give non-trivial equations, i.e. $ \text{tr}_3(E_1(i,j)E_1(x,y)) $, $ {tr}_3(E_1(i,j)E_2(x,y)) $, and $ \text{tr}_3(E_1(i,j)E_3(x,y)) $.

        First, we apply exchange relations to evaluate $ \text{tr}_3(E_1(i,j)E_1(x,y)) $.
        \begin{align}
            \text{tr}_3(E_1(i,j)E_1(x,y)) &= \delta_{j\overline{y}}N_{\overline{x}i}^{\overline{y}}\text{tr}_2(\overline{y}), \\
            \text{tr}_3(E_1(i,j)E_1(x,y)) &= \sum_{s,t=0}^{n} a_{st}^{ij}\text{tr}_3(E_2(s,t)E_1(x,y)) + \sum_{\ell,m=0}^{n} b_{\ell m}^{ij}\text{tr}_3(E_3(\ell,m)E_1(x,y)) \\
            &= \sum_{s=0}^{n} a_{sx}^{ij}N_{\overline{x}s}^{\overline{y}}\text{tr}_2(\overline{y}) + b_{\overline{y} x}^{ij}d_x\text{tr}_2(\overline{y}) \\
            (\text{by Eq. } \ref{coptrace})&= \sum_{s=0}^{n} a_{sx}^{ij}N_{x\overline{y}}^{s}\text{tr}_2(s) + \sum_{s=0}^{n}b_{\overline{y} x}^{ij}N_{x\overline{y}}^{s}\text{tr}_2(s) \\
            (\text{by Eq. } \ref{3face_1})&= \delta_{j\overline{y}}N_{x\overline{y}}^{i}\text{tr}_2(i)
        \end{align}
        Hence, two different ways give the same scalar by Frobenius reciprocity. Similarly, applying exchange relations to evaluate $ {tr}_3(E_1(i,j)E_2(x,y)) $ and $ \text{tr}_3(E_1(i,j)E_3(x,y)) $ also gives the same scalar by equations that we have obtained.

        Condition (2) in Theorem \ref{thm:conthm} reduces to basic consistency equation $ \delta = \sum\limits_{i=0}^{n} d_i $ in Proposition \ref{basic_cons}.
        
    \end{proof}

    \begin{remark}
        We call these equations the system of structure equations for exchange relation planar algebras. The highest degree of the system is $ 3 $. The number of variables is $ O(n^4) $, while the number of equations is $ O(n^6) $. A solution of the system produces a planar algebra. For comparison, the pentagon equation is also a system of polynomial equations of degree 3, involving several variables and integer coefficients, given a base fusion ring. A solution of the pentagon equation is the main condition for a fusion ring to be the Grothendieck ring of a unitary fusion category. The structure equations are more complicated than pentagon equations in that the fusion coefficients are unknown and assume values in non-negative real numbers. This complexity makes solving the structure equations a challenging task in the study of exchange relation planar algebras.
    \end{remark}

    \section{Forest fusion graph}\label{sec:cg}
    In this section, we introduce the weighted fusion graph of a fusion bialgebra, which encodes all information of a fusion bialgebra. This concept originates from an observation that the norm of $ T(i,j,k) $ is $ \sqrt{{N}_{kj}^{i}d_i\delta} $, so $ T(i,j,k) $ is zero if and only if $ {N}_{kj}^{i} = 0 \Leftrightarrow  \widetilde{N}_{kj}^{i} = 0 $. The zeroes of $ \{\widetilde{N}_{kj}^{i}\} $ play an important role in the classification of subfactorizable fusion bialgebras. 

    Theorem \ref{fub} shows the $ \mathbb{R}_{\ge 0} $-basis are scalar multiples of minimal projections of the 2-box space $ \mathcal{P}_{2,+} $. We adopt the fusion coefficients $ \{N_{kj}^{i}\}_{0\le i,j,k \le n} $ with respect to minimal projections $ \{p_{0}, p_{1}, \cdots, p_{n}\} $ for a fusion bialgebra $ \mathcal{B} $.

    \begin{definition}\label{def:cg}
        For each fusion matrix $ {N}_{k\cdot}^{\cdot} $ in a fusion bialgebra $ \mathcal{B} $, we define its weighted fusion graph $ \Gamma_k $ of $ p_k $ as a bipartite graph where white(black) vertices are row(column) indices. If $ N_{kj}^{i} \ne 0 $, then the white vertex $ i $ and the black vertex $ j $ are connected by a edge with the weight $ N_{kj}^{i} $. A weighted fusion graph $ \Gamma $ of the fusion bialgebra $ \mathcal{B} $ is a set of fusion graphs $ \{\Gamma_k: 0\le k \le n\} $. We call $ \Gamma $ a fusion graph if the weight is omitted.
    \end{definition}

    \begin{remark}
        A fusion bialgebra is equivalent to its weighted fusion graph. The weighted fusion graph of a subfactorizable fusion bialgebra describes the decomposition of convolutions of minimal projections. As a contrast, the principal graph of subfactors describes the decomposition of irreducible bimodule tensoring with the generating bimodule. The fusion graph is a significant combinatorial invariant for subfactorizable fusion bialgebras as the principal graph is for subfactors.
    \end{remark}

    \begin{example}
        The principal graphs of subfactor planar algebras of index less than 4 are classified by Dynkin diagrams of type $ A_n, D_{2n}, E_6 $, and $ E_8 $. The dimensions of 2-box spaces of these subfactor planar algebras are all 2. The fusion graph of $ \mathcal{P}_{2,+} $ contains two fusion graphs $ \Gamma_0 $ and $ \Gamma_1 $. $ \Gamma_0 $ is trivial for all fusion bialgebras. The fusion graph of $ p_1 $ is shown in the graph $ \Gamma_1 $. $ \Gamma_1 $ can also be the fusion graph of $ p_1 $ in Temperley-Lieb-Jones planar algebras $ TL(\delta) $ for $ \delta \ge \sqrt{2} $. When $ \delta = \sqrt{2} $, the edge connecting the white vertex 1 and the black vertex 1 is omitted.
        \begin{align}
            \Gamma_1: 
            \begin{tikzpicture}[scale=0.8,baseline={([yshift=-0.5ex]current bounding box.center)}]
                \foreach \x in {0,1}{
                  \node[circle, fill=black, inner sep=2pt, outer sep=1pt, label=left: \x] (A\x) at (0,-\x) {};
                }
                \foreach \y in {0,1}{
                  \node[circle,draw, inner sep=2pt, outer sep=1pt, label=right: \y] (B\y) at (2,-\y) {};
                }
                \draw (A0) -- (B1);
                \draw (A1) -- (B0);
                \draw (A1) -- (B1);
            \end{tikzpicture}.
        \end{align}
    \end{example}

    \begin{example}
        In the one-parameter BMW planar algebras, $ \dim\mathcal{P}_{2,\pm} = 3 $ and $ \dim \mathcal{P}_{3,\pm} = 14 $. $ p_1 $ and $ p_2 $ are both self-dual.  The fusion graph $ \Gamma $ is shown as follows.
        \begin{align}
            \Gamma_1: 
            \begin{tikzpicture}[scale=0.8,baseline={([yshift=-0.5ex]current bounding box.center)}]
                \foreach \x in {0,1,2}{
                  \node[circle, fill=black, inner sep=2pt, outer sep=1pt, label=left: \x] (A\x) at (0,-\x) {};
                }
                \foreach \y in {0,1,2}{
                  \node[circle,draw, inner sep=2pt, outer sep=1pt, label=right: \y] (B\y) at (2,-\y) {};
                }
                \draw (A0) -- (B1);
                \draw (A1) -- (B0);
                \draw (A1) -- (B2);
                \draw (A2) -- (B1);
                \draw (A2) -- (B2);
            \end{tikzpicture},
              \quad
            \Gamma_2: 
            \begin{tikzpicture}[scale=0.8,baseline={([yshift=-0.5ex]current bounding box.center)}]
                \foreach \x in {0,1,2}{
                  \node[circle, fill=black, inner sep=2pt, outer sep=1pt, label=left: \x] (A\x) at (0,-\x) {};
                }
                \foreach \y in {0,1,2}{
                  \node[circle,draw, inner sep=2pt, outer sep=1pt, label=right: \y] (B\y) at (2,-\y) {};
                }
                \draw (A0) -- (B2);
                \draw (A1) -- (B1);
                \draw (A1) -- (B2);
                \draw (A2) -- (B0);
                \draw (A2) -- (B1);
                \draw (A2) -- (B2);
            \end{tikzpicture}.
        \end{align}
    \end{example}

    \begin{example}
        The fusion graphs of $ p_1 $ and $ p_2 $ in Yang-Baxter relation planar algebras with $ \dim \mathcal{P}_{2,\pm} = 3 $ and $ \dim \mathcal{P}_{3,\pm} = 15 $.

        Case 1: $ p_1 $ and $ p_2 $ are both self-dual. This case corresponds to 2-parameter BMW planar algebras.
        \begin{align}
            \Gamma_1: 
            \begin{tikzpicture}[scale=0.8,baseline={([yshift=-0.5ex]current bounding box.center)}]
                \foreach \x in {0,1,2}{
                  \node[circle, fill=black, inner sep=2pt, outer sep=1pt, label=left: \x] (A\x) at (0,-\x) {};
                }
                \foreach \y in {0,1,2}{
                  \node[circle,draw, inner sep=2pt, outer sep=1pt, label=right: \y] (B\y) at (2,-\y) {};
                }
                \draw (A0) -- (B1);
                \draw (A1) -- (B0);
                \draw (A1) -- (B1);
                \draw (A1) -- (B2);
                \draw (A2) -- (B1);
                \draw (A2) -- (B2);
            \end{tikzpicture},
              \quad
            \Gamma_2: 
            \begin{tikzpicture}[scale=0.8,baseline={([yshift=-0.5ex]current bounding box.center)}]
                \foreach \x in {0,1,2}{
                  \node[circle, fill=black, inner sep=2pt, outer sep=1pt, label=left: \x] (A\x) at (0,-\x) {};
                }
                \foreach \y in {0,1,2}{
                  \node[circle,draw, inner sep=2pt, outer sep=1pt, label=right: \y] (B\y) at (2,-\y) {};
                }
                \draw (A0) -- (B2);
                \draw (A1) -- (B1);
                \draw (A1) -- (B2);
                \draw (A2) -- (B0);
                \draw (A2) -- (B1);
                \draw (A2) -- (B2);
            \end{tikzpicture}.
        \end{align}

        Case 2: $ p_1 $ and $ p_2 $ are dual to each other. This case corresponds to a new family of planar algebras constructed by the second author. 
        \begin{align}
            \Gamma_1: 
            \begin{tikzpicture}[scale=0.8,baseline={([yshift=-0.5ex]current bounding box.center)}]
                \foreach \x in {0,1,2}{
                  \node[circle, fill=black, inner sep=2pt, outer sep=1pt, label=left: \x] (A\x) at (0,-\x) {};
                }
                \foreach \y in {0,1,2}{
                  \node[circle,draw, inner sep=2pt, outer sep=1pt, label=right: \y] (B\y) at (2,-\y) {};
                }
                \draw (A0) -- (B1);
                \draw (A1) -- (B1);
                \draw (A1) -- (B2);
                \draw (A2) -- (B0);
                \draw (A2) -- (B1);
                \draw (A2) -- (B2);
            \end{tikzpicture},
              \quad
            \Gamma_2: 
            \begin{tikzpicture}[scale=0.8,baseline={([yshift=-0.5ex]current bounding box.center)}]
                \foreach \x in {0,1,2}{
                  \node[circle, fill=black, inner sep=2pt, outer sep=1pt, label=left: \x] (A\x) at (0,-\x) {};
                }
                \foreach \y in {0,1,2}{
                  \node[circle,draw, inner sep=2pt, outer sep=1pt, label=right: \y] (B\y) at (2,-\y) {};
                }
                \draw (A0) -- (B2);
                \draw (A1) -- (B1);
                \draw (A1) -- (B2);
                \draw (A1) -- (B0);
                \draw (A2) -- (B1);
                \draw (A2) -- (B2);
            \end{tikzpicture}
        \end{align}
    \end{example}

    \begin{lemma}\label{thm:acyclic}
        Suppose $ \mathcal{B} $ is a subfactorizable fusion bialgebra with exchange relations. If the fusion graph $ \Gamma_k $ of fusion matrix $ {N}_{k\cdot}^{\cdot} $ is acyclic, $ T(i,j,k) $ can be represented as an alternating sum of $ E_2 $-vectors and $ E_3 $-vectors. Furthermore, let $ C=C_1\sqcup (i,j)\sqcup C_2 $ be the connected component in which the edge $ (i,j) $ lies. $ C_1 $ ($ C_2 $) is the connected component in which $ i $ ($ j $) lies when removing the edge $ (i,j) $ from $ C $. We have the following two formulas.
        \begin{align} 
            T(i,j,k) = \sum_{\text{white}\  s\in C_1} E_2(s,k) - \sum_{\substack{\text{black}\  \ell \in C_1 }} E_3(\ell,k), \label{maineq1}\\ 
            T(i,j,k) = \sum_{\text{black}\  \ell \in C_2} E_3(\ell,k) - \sum_{\substack{\text{white}\  s \in C_2 }} E_2(s,k). \label{maineq2}
        \end{align}
    \end{lemma}
    \begin{proof}
        We first give an interpretation of the fusion graph $ \Gamma_k $. We use the white vertex $ s $ to represent $ E_2(s,k) $ and the black vertex $ \ell $ to represent $ E_3(\ell,k) $. The edge connecting the white vertex $ s $ represents a $ T $-vector that is a summand in $ E_2(s,k) $ by the following formula.
        \begin{align}
            E_2(s,k) = \sum_{\substack{0\le\ell\le n, \\ N_{k\ell}^{s} \ne 0}} T(s,\ell,k). \\
            E_3(\ell,k) = \sum_{\substack{0\le s\le n, \\ N_{k\ell}^{s} \ne 0}} T(s,\ell,k).
        \end{align}
        The edge $ (s,\ell) $ exists if and only if $ T(s,\ell,k) \ne 0 $ by definition. In this case, $ T(s,\ell,k) $ is a common summand of $ E_2(s,k) $ and $ E_3(\ell,k) $.
        
        If $ T(i,j,k) $ is zero, the edge $ (i,j) $ does not exist and the conclusion is obvious. If $ T(i,j,k) \ne 0 $, the white vertex $ i $ and the black vertex $ j $ are adjacent. Note that the sum of all $ E_2 $-vectors in $ C $ and the sum of all $ E_3 $-vectors in $ C $ are equal, because all edges in $ C $ are precisely all edges incident to white(black) vertices.

        Since $ \Gamma_k $ is acyclic, each of its connected component is a tree. If we remove the edge $ (i,j) $ from $ C $, $ C $ will be divided into two connected components $ C_1 $ and $ C_2 $. While all black vertices in $ C_1 $ represent $ E_3 $-vectors, all white vertices in $ C_1 $ represent $ E_2 $-vectors except that $ i $ misses the edge $ (i,j) $ corresponding to $ T(i,j,k) $. Counting edges in $ C_1 $ in two ways gives
        \begin{align}
            \sum_{\text{white}\  s\in C_1}E_2(s,k) - T(i,j,k) =  \sum_{\substack{\text{black}\  \ell \in C_1 }} E_3(\ell,k).
        \end{align} 
        
        Thus we obtain Formula \ref{maineq1}. Similar arguments apply to $ C_2 $ and we obtain Formula \ref{maineq2}. 
    \end{proof}

    \begin{theorem}\label{thm:forest}
        Suppose $ \mathcal{B} $ is a subfactorizable fusion bialgebra. $ \mathcal{B} $ has an exchange relation if and only if every fusion graph $ \Gamma_k $  of the fusion matrix $ {N}_{k\cdot}^{\cdot} $ is a forest.
    \end{theorem}
    \begin{proof}
        We first prove the ``if" part. Since 
        \begin{align}
            E_1(i,j) = \sum_{k=0}^{n} T(i,j,k) ,
        \end{align}
        and $ T(i,j,k) $ is an alternating sum of $ E_2 $ and $ E_3 $ vectors by Lemma \ref{thm:acyclic}, the $ E_1 $-vector is a linear combination of $ E_2 , E_3 $-vectors. Therefore, $ \mathcal{P} $ satisfies exchange relations.

        Now we prove the ``only if" part by contradiction. Suppose there exists a cycle $ \gamma $  on certain fusion graph $ \Gamma_{k_0} $. We fix an edge $ (i,j) $ in $ \gamma $, and label the coefficients of $ E_2 $-vectors and $ E_3 $-vectors on vertices of $ \gamma $ according to Equation \ref{labeq}. We neglect the edge $ (i,j) $ and traverse $ \gamma $ by starting from white vertex $ i $ and stopping at black vertex $ j $ finally. Each white vertex is labelled $ a $ while each black vertex is labelled $ -a $. However, the sum of labels on $ i $ and $ j $ should be $ 1 $, which is a contradiction. 
    \end{proof}

    \begin{remark}
        This theorem provides another proof of the fact that $ \dim \mathcal{P}_{3} \le (\dim \mathcal{P}_{2})^{2} + (\dim \mathcal{P}_{2} - 1)^{2} $. The forest structure indicates that $ v - e \ge 1 $ for each fusion graph. Thus the edges in the fusion graph, the number of which is $ \dim \mathcal{P}_{3} $, is less than or equal to $ n + (n-1)(2n-1) $, if $ \dim \mathcal{P}_{2} = n $. 
    \end{remark}

    \begin{definition}
        A forest fusion graph $ \Gamma $ is a fusion graph such that every $ \Gamma_k \in \Gamma $ is a forest. 
    \end{definition}

    \begin{example}
        A subfactorizable fusion bialgebra originates from a group if and only if each fusion graph of a minimal projection is a 1-regular graph. Hence, group planar algebras are the simplest planar algebras that have forest fusion graphs.
    \end{example}

    \begin{theorem}\label{thm:parameter}
        Given a forest fusion graph $ \Gamma $ in a subfactorizable fusion bialgebra with exchange relations, the parameters $ \{a_{sk}^{ij}\}_{0\le i,j,s,k \le n}$, $\{b_{\ell m}^{ij}\}_{0\le i,j,\ell,m \le n} $ in exchange relation \ref{exr2} are determined to be $ 0 $ or $ \pm 1 $ by $ \Gamma $.
    \end{theorem}
    \begin{proof}
        Given a forest fusion graph $ \Gamma $, multiplying Formula \ref{exr2} by $ p_k $ from below gives 
        \begin{align}
            T(i,j,k) = \sum_{s=0}^{n} a_{sk}^{ij}E_2(s,k) + \sum_{\ell = 0}^{n} b_{\ell k}^{ij}E_3(\ell,k).
        \end{align}
        We compare this formula with Formula \ref{maineq1}. To make ensure the right hand side of both formulas are formally the same, the coefficients $ a_{sk}^{ij}E_2(s,k) $ and $ b_{\ell k}^{ij} $ are determined to be $ 0 $ or $ \pm 1 $. Similar result holds for Formula \ref{maineq2}.

        Given a forest fusion graph $ \Gamma $, we set the canonical parameters according to Formula \ref{maineq1} and Formula \ref{maineq2}.
        
        If $ T(i,j,k) = 0 $, set $ a_{sk}^{ij} = 0, b_{sk}^{ij} = 0 $ for $ 0 \le s \le n $.

        If $ T(i,j,k) \ne 0 $, there are $ 2 $ cases.
        \begin{itemize}
            \item $ k = 0 $. $ i $ and $ j $ should be the same. Set $ a_{i0}^{ii} = 1, b_{i0}^{ii} = 0 $ for $ 0 \le i \le n $.
            \item $ k \ne 0 $.
            \begin{itemize}
                \item  If either $ i $ or $ j $ is 0, the other should be equal to $ k $ or $ \overline{k} $. Set $ a_{0k}^{0\overline{k}} = 1, b_{kk}^{0\overline{k}} = 0, a_{kk}^{k0} = 0, b_{0k}^{k0} = 1 $ for $ 1 \le k \le n $. 
                \item If both $ i $ and $ j $ are non-zero, set $ a_{sk}^{ij} = 1, b_{\ell k}^{ij} = -1 $ for $ s, \ell $ in Formula \ref{maineq1}.
            \end{itemize} 
        \end{itemize}
        All other parameters are set to $ 0 $.
    \end{proof}

    \begin{theorem}
        Given a forest fusion graph, the fusion coefficient $ N_{ij}^{k} $ is a sum of $ d_i $'s with positive or negative signs.
    \end{theorem}

    \begin{proof}
        The conclusion follows from Theorem \ref{thm:parameter} and Formula \ref{E_2face_1}.
    \end{proof}

    \begin{remark}
        Given a forest fusion graph, the exchange relations are precisely determined, enabling us to deduce the parameters directly from the graphs. The essential variables of structure equations are quantum dimensions $ \{d_i\}_{0\le i\le n} $. This represents a significant reduction in complexity, decreasing from \( O(n^4) \) to just \( n + 1 \). Consequently, the structure equations can be expressed as linear and quadratic equations on these variables. The number of linear equations far exceeds the number of variables. Thus it is easy to solve structure equations for each forest fusion graph by first solving linear equations and then verifying quadratic equations. 
    \end{remark}
    
    From the perspective of algebraic geometry, we decompose a complicated algebraic variety into many simpler subvarieties with forest structure. Now the complexity of exchange relations transits from analyzing a complicated algebraic variety to finding out all forest fusion graphs.

    Note that the total number of structure equations is $ O(n^6) $, so this is an overdetermined system, which has no solution in general. This indicates the presence of obstructions to the solutions of structure equations, reflecting the inherent rigidity of subfactors.

    \section{Classification of exchange relation planar algebras}

    The classification problem has been significantly simplified, however, the exponential number of forest fusion graphs poses a considerable challenge to humans. Therefore, we need computers to help enumerate forest fusion graphs. We give an equivalent definition of fusion graph, the indicator function of fusion coefficients. The advantage of indicator function is that it can be encoded into a binary string, thereby facilitating our computer implementation.

    \subsection{Indicator function of fusion coefficients}
    Let $ L_2 = \{p_i\}_{0\le i\le n} $ be the generator set and identify $ p_i $ with $ i $. The fusion coefficients $ N_{kj}^{i} $ can be viewed as a 3-tensor with $ L_2 $ as the index set. We use the indicator functions to record the zeroes of fusion coefficients.
    \begin{definition}
        Let $ \varphi: L_2\times L_2\times L_2\to \{0,1\} $ be the indicator function of $ L_2\times L_2\times L_2 $. We use $ \varphi(p_i,p_j,p_k) $ to indicate whether $ N_{kj}^{i} $ is zero or not, which corresponds to whether the edge $ (i,j) $ in $ \Gamma_k $ exists or not.
    \end{definition}

    \begin{definition}
        A unital dual-preserving automorphism $ g $ on the generator set $ L_2 $ is a bijection such that
        \begin{itemize}
            \item $ g(\overline{p}) = \overline{g(p)} $ for any $ p\in L_2 $,
            \item $ g(p_0) = p_0 $.
        \end{itemize} 
    \end{definition}

    \begin{proposition}
        Suppose there are $ m $ pairs of dual minimal projections in $ L_2 = \{p_i\}_{0\le i\le n} $. The group $ G $ of unital dual-preserving automorphisms on $ L_2 $ is isomorphic to $ S_m\times \mathbb{Z}_2^{\oplus m}\times S_{n-2m-1} $.
    \end{proposition}
    \begin{proof}
        Let $ g\in G $ and identify $ p_i $ with $ i $, then $ g $ must fix $ 0 $. $ g $ should map a pair of dual projections to a pair of dual projections and map a self-dual projection to a self-dual projection. When $ G $ acts on self-dual projections, it is isomorphic to the permutation group on $ n-2m-1 $ objects. When $ g $ acts on $ m $ pairs of dual projections, it is a permutation on $ m $ objects if regarding a pair of dual projections as an object. However, for each $ p $ in one pair, $ g(p) $ has two choices in another pair. Thus the conclusion follows.
    \end{proof}

    \begin{definition}
        Let $ \varphi $ be an indicator function on $ L_2\times L_2\times L_2 $. 
        The group $ G $ of unital dual-preserving automorphisms acts naturally on $ \varphi $ by $ (g\varphi)(i,j,k) := \varphi(g^{-1}i,g^{-1}j,g^{-1}k) $. Two indicator functions are called isomorphic if they are in the same $ G $-orbit.
    \end{definition}

    \begin{definition}
        We define an equivalence relation on $ L_2\times L_2\times L_2 $ by $ (i_1,j_1,k_1) \sim (i_2,j_2,k_2) $ if $ N_{k_1j_1}^{i_1} = 0 \Leftrightarrow N_{k_2j_2}^{i_2} = 0 $. An indication function on $ L_2\times L_2\times L_2 $ is called admissible if it is compatible with the equivalence relation induced by Frobenius reciprocity, unital structure, and possible commutativity.
    \end{definition}
    \begin{remark}
        The unital structure means a fusion coefficient $ N_{kj}^{i} $ is clear if either one of $ i,j,k $ is $ 0 $.
        An admissible indicator function on $ L_2\times L_2\times L_2 $ is naturally an indicator function on $ L_2\times L_2\times L_2/\sim $. Note that $ G $-action is also compatible with Frobenius reciprocity and commutativity, so $ G $ acts naturally on $ L_2\times L_2\times L_2/\sim $, and thus on admissible indicator functions.
    \end{remark}

    \begin{proposition}
        Let $ \mathcal{B} $ be a fusion bialgebra having $ \{N_{kj}^{i}\}_{0\le i,j,k\le n} $ as fusion coefficients of $ L_2 $. Then for $ g\in G $, $ L_2 $ with fusion coefficients $ \{N_{g(k) g(j)}^{g(i)}\} $ is also a fusion bialgebra $ \mathcal{B}_g $ isomorphic to $ \mathcal{B} $. Furthermore, if $ \mathcal{B} $ is subfactorizable and has exchange relations, so is $ \mathcal{B}_g $.
    \end{proposition}
    \begin{proof}
        Mapping $ p_k $ in $ \mathcal{B}_g $ to $ g(p_k) $ in $ \mathcal{B} $ gives an isomorphism between $ \mathcal{B}_g $ and $ \mathcal{B} $. Theorem 2.26 in \cite{Liu2016Exchange} asserts that this bialgebra isomorphism will give a planar algebra isomorphism.
    \end{proof}
    \begin{remark}
        This proposition implies that we only need to select one indicator function in each $ G $-orbit to give a forest fusion graph and solve the structure equations. 
    \end{remark}

    By definition of indicator functions, there are $ 2^{m} $ admissible indicator functions in total if $ |L_2\times L_2\times L_2/\sim| = m $. Despite of the exponential growth, $ G $-action and forest structure significantly reduces the number, when $ \dim \mathcal{P}_{2,\pm} $ is small. Note that the value of $ N_{kj}^{i} $ is clear if one of $ i,j,k $ is 0, so we do not need to consider such fusion coefficients.

    \subsection{Analytic criteria}

    Given a fusion bialgebra, we propose two novel analytic criteria to detect whether it originates from an exchange relation planar algebra. The criteria are called analytic because they employ the positivity of fusion coefficients, which reflects the unitary property of subfactors.

    According to Equation \ref{associativity}, we define the following formula:
    \begin{align}
        f(i,j,k,\ell) = \sum_{s=0}^{n} N_{ij}^{s}N_{sk}^{\ell} - \sum_{s=0}^{n} N_{is}^{\ell}N_{jk}^{s}.
    \end{align}
    View $ f $ as a multivariate polynomial on fusion coefficients, we have the following positivity criterion.
    \begin{proposition}[Associative Positivity Criterion]
        Given an indicator function $ \varphi $ of a fusion bialgebra $ \mathcal{B} $, if there exists $ i,j,k,\ell $ such that  every monomial in $ f(i,j,k,\ell) $ has the same sign, then $ \mathcal{B} $ is not subfactorizable.
    \end{proposition}
    \begin{proof}
        If every monomial in $ f(i,j,k,\ell) $ has the same sign, then the equation $ f(i,j,k,\ell) = 0 $ can only have zero solutions due to positivity of fusion coefficients. However, this contradicts to the assumption on the indicator function $ \varphi $.
    \end{proof}

    \begin{definition}
        Let $ \mathcal{B} $ be a subfactorizable fusion bialgebra with minimal projections $ L_2 $. We define an antisymmetric and transitive relation $ \leqq $ on $ L_2 $ such that $ p_i \leqq p_j $ if $ p_i * p_j $ and $ p_j * p_i $ is are scalar multiples of $ p_j $. 
    \end{definition}

    \begin{theorem}\label{thm:FreeProduct}
        Suppose $ \mathcal{B} $ is a subfactorizable fusion bialgebra having minimal projections $ L_2 $ as generators of the corresponding subfactor planar algebra $ \mathcal{P} $.
        Then $ \mathcal{B} $ originates from a free product of subfactor planar algebras if and only if $ L_2 $ can be partitioned into two non-empty subsets $ S $ and $ R $ such that $ p_i \leqq p_k $ for any $ p_i \in S, p_k \in R $.
    \end{theorem}
    \begin{proof}
        The ``only if" part is easy. In this case, there exists a non-trivial biprojection $ Q $ such that $ QxQ = x $ or $ Q*x*Q = (\frac{\text{tr}(Q)}{\delta})^2 x $ for $ x \in L_2 $. Then it suffices to let $ S $ be the subset of $ L_2 $ such that $ QxQ = x $, and let $ R $ be the subset of $ L_2 $ such that $ Q*x*Q = (\frac{\text{tr}(Q)}{\delta})^2 x $.
        
        Now we prove the ``if" part. 
        \textbf{Claim:} The convolution is closed in $ S $. 

        Suppose not. Then there exists $ p_i, p_j \in S $ such that $ p_i*p_j $ contains $ p_r $ where $ p_r \in R $. Then $ p_{\overline{r}} \in R $ since $ p_{\overline{r}} \nleqq p_r $. Note that $ p_i*p_j*p_{\overline{r}} = C_1p_{\overline{r}} $ and $ p_{\overline{r}}*p_i*p_j = C_2p_{\overline{r}} $ where $ C_1,C_2 $ are scalars. This implies that $ p_r \leqq p_{\overline{r}} $, which is a contradiction. 

        Let $ Q = \sum_{p_i\in S}^{} p_i  $. Then $ Q $ is a biprojection. Furthermore, $ QxQ = x $ for $ x \in S $ and $ Q*x*Q = (\frac{\text{tr}(Q)}{\delta})^2 x $ for $ x \in R $. Thus $ \mathcal{B} $ originates from a free product of subfactor planar algebras.
    \end{proof}

    We provide Algorithm \ref{algorithm:genS} to generate the set $ S $. First, initialize two empty sets $ S $ and $ R $.
    Put $ p_0 $ in $ S $. 
    Second, find all $ p_i $ such that if $ p_j \leqq p_i $, then $ j = 0 $. Put such $ p_i $ in S.
    Third, for $ x\in L_2\setminus S $, check whether $ p_i \leqq x $ for any $ p_i\in S $. If so, put $ x $ in $ R $. Otherwise, put $ x $ in $ S $, and move $ p_r $ from $ R $ to $ S $ if $ x \nleqq p_r $. Finally, we obtain a set $ S $ such that $ p_i \leqq p_k $ for any $ p_i \in S, p_k \in L_2\setminus S $.

    \begin{algorithm}
        \caption{Generate a minimum $ S $ from $ L_2 $}\label{algorithm:genS}
        \begin{algorithmic}[1]
            \Statex \textbf{Input:} $ L_2 $ with a forest fusion graph $ \Gamma $.
            \Statex \textbf{Output:} A set $ S $.
            \State $ n = |L_2|-1 $
            \State Initialize two empty sets $ S $ and $ R $
            \State $ S $.insert($ p_0 $)  \Comment{Put $ p_0 $ in $ S $}

            \For{$ i=1,\dots, n $}
            \State Initialize an auxiliary empty set $ A_i $    
            \EndFor

            \For{$ i=1,\dots, n $} \Comment{Fill $ A_i $ with elements that are $ \leqq p_i $}
                \For{$ j=i,\dots, n $}
                \If{$ p_i \leqq p_j $}
                \State $ A_j $.insert($ p_i $)
                \ElsIf{$ p_j \leqq p_i $}
                \State $ A_i $.insert($ p_j $)
                \EndIf
                \EndFor
            \EndFor

            \For{$ i=1,\dots, n $} \Comment{Fill $ S $ with minimal elements}
                \If{$ |A_i| = 0 $}
                \State $ S $.insert($ p_i $)
                \EndIf
            \EndFor

            \For{$ p_k \in L_2\backslash S $} \Comment{Minimal extension for $ S $}
                \If{$ S\subset A_k $}
                \State $ R $.insert($ p_k $)
                \Else{ $ S $.insert($ p_k $)}
                \For{$ p_r \in R $}
                \If{ $ p_k\notin A_r $ }
                \State $ S $.insert($ p_r $)
                \EndIf
                \EndFor
                \EndIf
            \EndFor
        \end{algorithmic}
    \end{algorithm}

    \begin{theorem}[Free Product Criterion]
        $ \mathcal{B} $ originates from a free product of subfactor planar algebras if and only if $ L_2\setminus S $ is non-empty.
    \end{theorem}
    \begin{proof}
        The ``if" part follows from Theorem \ref{thm:FreeProduct}. The ``only if" part follows from the fact that our algorithm generates the set $ S $ with minimum cardinality.
    \end{proof}

    \begin{remark}
        This is a universal criterion for detecting the free product of subfactor planar algebras generated by 2-boxes. If $ \mathcal{P}*\mathcal{Q} $ is an exchange relation planar algebra, then both $ \mathcal{P} $ and $ \mathcal{Q} $ must be exchange relation planar algebras, and the converse is also true. Thus one can determine from a forest fusion graph whether a subfactorizable fusion bialgebra $ \mathcal{B} $ with exchange relations arises from a free product of two exchange relation planar algebras. In the free product case, $ \mathcal{B} $ is clear if the classification for fewer generators is already known.
    \end{remark}

    \begin{proposition}[Tensor Product Criterion]
        Let $ \Gamma $ be a forest fusion graph. If there exists a group $ H $ and a forest fusion graph $ \widetilde{\Gamma} $ for a subfactorizable fusion bialgebra with exchange relations, such that $ \Gamma \cong \Gamma_H\otimes\widetilde{\Gamma} $ where $ \otimes $ is the matrix tensor product for fusion matrices, then $ \Gamma $ originates from a subfactorizable fusion bialgebra with exchange relations.
    \end{proposition}
    \begin{proof}
        It follows from Proposition 5.5 in \cite{Liu2016Exchange}.
    \end{proof}

    The Tensor Product Criterion is trivial. We build a database for known examples of subfactorizable fusion bialgebras with exchange relations. If the dimension is a composite number, we construct all possible forest fusion graphs from tensor product construction and sieve graphs that are isomorphic to them according to above proposition.

\newpage

    \subsection{A delicate classification scheme}

    Based on forest fusion graph and analytic criteria, we propose an delicate classification scheme on subfactorizable fusion bialgebras with exchange relations. We anticipate this scheme will help explore new examples of exchange relation planar algebras.

    The proposed classification scheme consists of the following steps:
    \begin{enumerate}[label={\textbf{Step} $\mathbf{\arabic*}$.},leftmargin=9ex]
        \item List complete Consistency Equations.  (Theorem \ref{thm:evathm}, \ref{thm:conthm}, \ref{thm:polyeq}).
        \item Simplify Consistency Equations for every forest fusion graph. (Theorem \ref{thm:parameter}).
        \item \textbf{(FF)} Enumerate all Forest Fusion graphs.
        \item \textbf{(RG)} Select representative fusion graphs up to graph isomorphism.
        \item \textbf{(APC)} Sieve graphs that do not satisfy the Associative Positivity Criterion.
        \item \textbf{(FTPC)} Sieve graphs from the Free/Tensor Product of two planar algebras.
        \item Solve structure equations for each remaining fusion graph. If a solution exists, verify the unitarity condition to confirm that it arises from a subfactor planar algebra.
    \end{enumerate}

    Theoretical results in Step 1 and 2 have been established. Step 3 to 7 have been implemented by our automated computer program \textbf{FuBi} except the verification for unitarity. 
    
    In Step 3, We encode admissible indicator functions, which are equivalent to fusion graphs, into binary strings. \textbf{FuBi} will traverse over all possible binary strings to single out forest fusion graphs. From Step 4 to 6, \textbf{FuBi} applies various criteria to sieve graphs based on the binary string. For every remaining forest fusion graph, \textbf{FuBi} solve its corresponding structure equations derived from Step 2. If a solution exists, the consistency condition will be satisfied, and we need to verify the unitarity condition manually.

    \subsection{Classification for 5-dimension}

    In this section, we follow our classification program to classify subfactorizable fusion bialgebras $ \mathcal{B} $ with $ \dim \mathcal{B} = 5 $. We first need to determine the number of duals in  $ L_2=\{p_{0}, p_{1}, p_2, p_3, p_4 \} $. There are three different cases: 
    \begin{enumerate}
        \item[(1)] $ \overline{p_{1}}=p_{1}, \overline{p_{2}}=p_{2}, \overline{p_{3}}=p_{3}, \overline{p_{4}}=p_{4} $;
        \item[(2)] $ \overline{p_{1}}=p_{2}, \overline{p_{3}}=p_{3}, \overline{p_{4}}=p_{4} $;
        \item[(3)] $ \overline{p_{1}}=p_{2}, \overline{p_{3}}=p_{4} $.
    \end{enumerate}

    In case (1), convolution is commutative since
    \begin{align*}
         p_i*p_j=\sum\limits_{k=1}^{m} N_{i,j}^{k}p_k=\sum\limits_{k=1}^{m} N_{i,j}^{k}\overline{p_k}=\overline{(p_i*p_j)}=p_j*p_i
    \end{align*}
    for any $ 1 \le i,j \le 4 $.

    \begin{proposition}
        Convolution is commutative in case (3).
    \end{proposition}
    \begin{proof}
        Without loss of generality, it suffices to verify $ p_1*p_2=p_2*p_1 $ and $ p_1*p_3=p_3*p_1 $ to prove convolution is commutative.

    Consider the associativity equation $ (p_1*p_2)*p_1=p_1*(p_2*p_1) $. The fusion coefficients of $ p_1 $ on both sides give the following equation: 
    \begin{equation}\label{eq1}
        \begin{aligned}
        N_{1,2}^{0}N_{0,1}^{1}+N_{1,2}^{1}N_{1,1}^{1}+N_{1,2}^{2}N_{2,1}^{1}+N_{1,2}^{3}N_{3,1}^{1}+N_{1,2}^{4}N_{4,1}^{1}=&N_{2,1}^{0}N_{1,0}^{1}+N_{2,1}^{1}N_{1,1}^{1}+N_{2,1}^{2}N_{1,2}^{1}\\
        &+N_{2,1}^{3}N_{1,3}^{1}+N_{2,1}^{4}N_{1,4}^{1}.
    \end{aligned}
    \end{equation}

    According to Frobenius reciprocity, we have
    \begin{align*}
        N_{1,1}^{1}=N_{1,2}^{2}=N_{2,1}^{2}&=N_{2,2}^{2}=N_{2,1}^{1}=N_{1,2}^{1}, \\
        N_{1,3}^{1}d_1=N_{3,2}^{2}d_2=N_{2,1}^{4}d_4&=N_{4,2}^{2}d_2=N_{2,1}^{3}d_3=N_{1,4}^{1}d_1, \\
        N_{3,1}^{1}d_1=N_{1,2}^{4}d_4=N_{2,3}^{2}d_2&=N_{2,4}^{2}d_2=N_{4,1}^{1}d_1=N_{1,2}^{3}d_3.
    \end{align*}
    So Equation \ref{eq1} reduces to $ N_{1,2}^{3}=N_{2,1}^{3} $. Hence, $ N_{1,2}^{4} = N_{2,1}^{4} $. So the coefficients of $ p_i $ for $ 0 \le i \le 4 $ on both sides of $ p_1*p_2 $ and $ p_2*p_1 $ are equal, and we obtain $ p_1*p_2=p_2*p_1 $.

    Furthermore, we observe that $ N_{1,3}^{1}=N_{3,1}^{1} $, and Frobenius reciprocity gives $ N_{1,3}^{2}=N_{3,1}^{2} $ and $ N_{1,3}^{4}=N_{3,1}^{4} $. Consider the associativity equation $ (p_3*p_4)*p_3=p_3*(p_4*p_3) $. The fusion coefficients of $ p_3 $ on both sides give the following equation: 
    \begin{equation}\label{eq2}
        \begin{aligned}
            N_{3,4}^{0}N_{0,3}^{3}+N_{3,4}^{1}N_{1,3}^{3}+N_{3,4}^{2}N_{2,3}^{3}+N_{3,4}^{3}N_{3,3}^{3}+N_{3,4}^{4}N_{4,3}^{3}=&N_{4,3}^{0}N_{3,0}^{3}+N_{4,3}^{1}N_{3,1}^{3}+N_{4,3}^{2}N_{3,2}^{3}\\
            &+N_{4,3}^{3}N_{3,3}^{3}+N_{4,3}^{4}N_{3,4}^{3}.
        \end{aligned}
    \end{equation}
    
    According to Frobenius reciprocity, we have 
    \begin{align*}
        N_{3,4}^{1}d_1=N_{4,2}^{4}d_4=N_{2,3}^{3}d_3&=N_{3,4}^{2}d_2=N_{4,1}^{4}d_4=N_{1,3}^{3}d_3, \\
        N_{4,3}^{1}d_1=N_{3,2}^{3}d_3=N_{2,4}^{4}d_4&=N_{4,3}^{2}d_2=N_{3,1}^{3}d_3=N_{1,4}^{4}d_4, \\
        N_{3,3}^{3}=N_{3,4}^{4}=N_{4,3}^{4}&=N_{4,4}^{4}=N_{4,3}^{3}=N_{3,4}^{3}.
    \end{align*}
    So Equation \ref{eq2} reduces to $ N_{1,3}^{3}=N_{3,1}^{3} $, and we obtain $ p_1*p_3=p_3*p_1 $. Therefore, convolution is commutative.
    \end{proof}

    Now we determine the admissible indicator functions for each case according to Frobenius reciprocity and commutativity. We list the equivalence classes of fusion coefficients and group actions on them. Fusion coefficients concerning $ p_0 $ are omitted.

    \subsubsection{Equivalence classes of fusion coefficients in case (1)}

    $$t_{1} = \{N_{1,1}^{1}\},$$
    $$t_{2} = \{N_{1,2}^{1},N_{2,1}^{1},N_{1,1}^{2}\},$$
    $$t_{3} = \{N_{1,3}^{1},N_{3,1}^{1},N_{1,1}^{3}\},$$
    $$t_{4} = \{N_{1,4}^{1},N_{4,1}^{1},N_{1,1}^{4}\},$$
    $$t_{5} = \{N_{2,2}^{1},N_{2,1}^{2},N_{1,2}^{2}\},$$
    $$t_{6} = \{N_{2,3}^{1},N_{3,1}^{2},N_{1,2}^{3},N_{3,2}^{1},N_{2,1}^{3},N_{1,3}^{2}\},$$
    $$t_{7} = \{N_{2,4}^{1},N_{4,1}^{2},N_{1,2}^{4},N_{4,2}^{1},N_{2,1}^{4},N_{1,4}^{2}\},$$
    $$t_{8} = \{N_{3,3}^{1},N_{3,1}^{3},N_{1,3}^{3}\},$$
    $$t_{9} = \{N_{3,4}^{1},N_{4,1}^{3},N_{1,3}^{4},N_{4,3}^{1},N_{3,1}^{4},N_{1,4}^{3}\},$$
    $$t_{10} = \{N_{4,4}^{1},N_{4,1}^{4},N_{1,4}^{4}\},$$
    $$t_{11} = \{N_{2,2}^{2}\},$$
    $$t_{12} = \{N_{2,3}^{2},N_{3,2}^{2},N_{2,2}^{3}\},$$
    $$t_{13} = \{N_{2,4}^{2},N_{4,2}^{2},N_{2,2}^{4}\},$$
    $$t_{14} = \{N_{3,3}^{2},N_{3,2}^{3},N_{2,3}^{3}\},$$
    $$t_{15} = \{N_{3,4}^{2},N_{4,2}^{3},N_{2,3}^{4},N_{4,3}^{2},N_{3,2}^{4},N_{2,4}^{3}\},$$
    $$t_{16} = \{N_{4,4}^{2},N_{4,2}^{4},N_{2,4}^{4}\},$$
    $$t_{17} = \{N_{3,3}^{3}\},$$
    $$t_{18} = \{N_{3,4}^{3},N_{4,3}^{3},N_{3,3}^{4}\},$$
    $$t_{19} = \{N_{4,4}^{3},N_{4,3}^{4},N_{3,4}^{4}\},$$
    $$t_{20} = \{N_{4,4}^{4}\}.$$
    Identify $ t_i $ with $ i $, then the $ S_4 $-action on equivalence classes is given in the form of permutations on $ t_i $'s. The trivial action is omitted.

    $$( 3,4 )( 6,7 )( 8,10 )( 12,13 )( 14,16 )( 17,20 )( 18,19 ),$$
$$( 2,3 )( 5,8 )( 7,9 )( 11,17 )( 12,14 )( 13,18 )( 16,19 ),$$
$$( 2,3,4 )( 5,8,10 )( 6,9,7 )( 11,17,20 )( 12,18,16 )( 13,14,19 ),$$
$$( 2,4,3 )( 5,10,8 )( 6,7,9 )( 11,20,17 )( 12,16,18 )( 13,19,14 ),$$
$$( 2,4 )( 5,10 )( 6,9 )( 11,20 )( 12,19 )( 13,16 )( 14,18 ),$$
$$(1,11 )( 2,5 )( 3,12 )( 4,13 )( 8,14 )( 9,15 )( 10,16 ),$$
$$(1,11 )( 2,5 )( 3,13 )( 4,12 )( 6,7 )( 8,16 )( 9,15 )( 10,14 )( 17,20 )( 18,19 ),$$
$$(1,11,17 )( 2,12,8 )( 3,5,14 )( 4,13,18 )( 7,15,9 )( 10,16,19 ),$$
$$(1,11,17,20 )( 2,12,18,10 )( 3,13,8,16 )( 4,5,14,19 )( 6,15,9,7 ),$$
$$(1,11,20,17 )( 2,13,19,8 )( 3,5,16,18 )( 4,12,10,14 )( 6,7,15,9 ),$$
$$(1,11,20 )( 2,13,10 )( 3,12,19 )( 4,5,16 )( 6,15,9 )( 8,14,18 ),$$
$$(1,17,11 )( 2,8,12 )( 3,14,5 )( 4,18,13 )( 7,9,15 )( 10,19,16 ),$$
$$(1,17,20,11 )( 2,8,19,13 )( 3,18,16,5 )( 4,14,10,12 )( 6,9,15,7 ),$$
$$(1,17 )( 2,14 )( 3,8 )( 4,18 )( 5,12 )( 7,15 )( 10,19 ),$$
$$(1,17,20 )( 2,14,16 )( 3,18,10 )( 4,8,19 )( 5,12,13 )( 6,15,7 ),$$
$$(1,17 )( 2,18 )( 3,8 )( 4,14 )( 5,19 )( 6,9 )( 7,15 )( 10,12 )( 11,20 )( 13,16 ),$$
$$(1,17,11,20 )( 2,18,5,19 )( 3,14,13,10 )( 4,8,12,16 )( 6,15,7,9 ),$$
$$(1,20,17,11 )( 2,10,18,12 )( 3,16,8,13 )( 4,19,14,5 )( 6,7,9,15 ),$$
$$(1,20,11 )( 2,10,13 )( 3,19,12 )( 4,16,5 )( 6,9,15 )( 8,18,14 ),$$
$$(1,20,17 )( 2,16,14 )( 3,10,18 )( 4,19,8 )( 5,13,12 )( 6,7,15 ),$$
$$(1,20 )( 2,16 )( 3,19 )( 4,10 )( 5,13 )( 6,15 )( 8,18 ),$$
$$(1,20,11,17 )( 2,19,5,18 )( 3,10,13,14 )( 4,16,12,8 )( 6,9,7,15 ),$$
$$(1,20 )( 2,19 )( 3,16 )( 4,10 )( 5,18 )( 6,15 )( 7,9 )( 8,13 )( 11,17 )( 12,14 ).$$

    \subsubsection{Equivalence classes of fusion coefficients in case (2)}
    $$t_{1} = \{N_{1,1}^{1} , N_{1,2}^{2} , N_{2,1}^{2} , N_{2,2}^{2} , N_{2,1}^{1} , N_{1,2}^{1}\},$$
    $$t_{2} = \{N_{1,3}^{1} , N_{3,2}^{2} , N_{2,1}^{3}\},$$
    $$t_{3} = \{N_{1,4}^{1} , N_{4,2}^{2} , N_{2,1}^{4}\},$$
    $$t_{4} = \{N_{2,2}^{1} , N_{1,1}^{2}\},$$
    $$t_{5} = \{N_{2,3}^{1} , N_{3,2}^{1} , N_{2,2}^{3} , N_{3,1}^{2} , N_{1,1}^{3} , N_{1,3}^{2}\},$$
    $$t_{6} = \{N_{2,4}^{1} , N_{4,2}^{1} , N_{2,2}^{4} , N_{4,1}^{2} , N_{1,1}^{4} , N_{1,4}^{2}\},$$
    $$t_{7} = \{N_{3,1}^{1} , N_{1,2}^{3} , N_{2,3}^{2}\},$$
    $$t_{8} = \{N_{3,3}^{1} , N_{3,2}^{3} , N_{2,3}^{3} , N_{3,3}^{2} , N_{3,1}^{3} , N_{1,3}^{3}\},$$
    $$t_{9} = \{N_{3,4}^{1} , N_{4,2}^{3} , N_{2,3}^{4} , N_{4,3}^{2} , N_{3,1}^{4} , N_{1,4}^{3}\},$$
    $$t_{10} = \{N_{4,1}^{1} , N_{1,2}^{4} , N_{2,4}^{2}\},$$
    $$t_{11} = \{N_{4,3}^{1} , N_{3,2}^{4} , N_{2,4}^{3} , N_{3,4}^{2} , N_{4,1}^{3} , N_{1,3}^{4}\},$$
    $$t_{12} = \{N_{4,4}^{1} , N_{4,2}^{4} , N_{2,4}^{4} , N_{4,4}^{2} , N_{4,1}^{4} , N_{1,4}^{4}\},$$
    $$t_{13} = \{N_{3,3}^{3}\},$$
    $$t_{14} = \{N_{3,4}^{3} , N_{4,3}^{3} , N_{3,3}^{4}\},$$
    $$t_{15} = \{N_{4,4}^{3} , N_{4,3}^{4} , N_{3,4}^{4}\},$$
    $$t_{16} = \{N_{4,4}^{4}\}.$$

    The $ Z_2\times S_2 $-action is given in the form of permutations of $ t_i $'s.
    $$( 2,3 )( 5,6 )( 7,10 )( 8,12 )( 9,11 )( 13,16 )( 14,15 ),$$
    $$( 2,7 )( 3,10 )( 9,11 ),$$
    $$( 2,10 )( 3,7 )( 5,6 )( 8,12 )( 13,16 )( 14,15 ).$$

    \subsubsection{Equivalence classes of fusion coefficients in case (3)}
    $$t_{1} = \{N_{1,1}^{1} , N_{1,2}^{2} , N_{2,1}^{2} , N_{2,2}^{2} , N_{2,1}^{1} , N_{1,2}^{1}\},$$
    $$t_{2} = \{N_{1,3}^{1} , N_{3,2}^{2} , N_{2,1}^{4} , N_{4,2}^{2} , N_{2,1}^{3} , N_{1,4}^{1} , N_{3,1}^{1} , N_{2,3}^{2} , N_{1,2}^{4} , N_{2,4}^{2} , N_{1,2}^{3} , N_{4,1}^{1}\},$$
    $$t_{3} = \{N_{2,2}^{1} , N_{1,1}^{2}\},$$
    $$t_{4} = \{N_{2,3}^{1} , N_{3,2}^{1} , N_{2,2}^{4} , N_{4,1}^{2} , N_{1,1}^{3} , N_{1,4}^{2}\},$$
    $$t_{5} = \{N_{2,4}^{1} , N_{4,2}^{1} , N_{2,2}^{3} , N_{3,1}^{2} , N_{1,1}^{4} , N_{1,3}^{2}\},$$
    $$t_{6} = \{N_{3,3}^{1} , N_{3,2}^{4} , N_{2,3}^{4} , N_{4,4}^{2} , N_{4,1}^{3} , N_{1,4}^{3}\},$$
    $$t_{7} = \{N_{3,4}^{1} , N_{4,2}^{4} , N_{2,3}^{3} , N_{3,4}^{2} , N_{4,1}^{4} , N_{1,3}^{3} , N_{4,3}^{1} , N_{2,4}^{4} , N_{3,2}^{3} , N_{4,3}^{2} , N_{1,4}^{4} , N_{3,1}^{3}\},$$
    $$t_{8} = \{N_{4,4}^{1} , N_{4,2}^{3} , N_{2,4}^{3} , N_{3,3}^{2} , N_{3,1}^{4} , N_{1,3}^{4}\},$$
    $$t_{9} = \{N_{3,3}^{3} , N_{3,4}^{4} , N_{4,3}^{4} , N_{4,4}^{4} , N_{4,3}^{3} , N_{3,4}^{3}\},$$
    $$t_{10} = \{N_{4,4}^{3} , N_{3,3}^{4}\}.$$

    The $ S_2\times Z_2^{\oplus 2} $-action is given in the form of permutations of $ t_i $'s.
    $$( 4,5 )( 6,8 ),$$
    $$(1,9 )( 2,7 )( 3,10 )( 4,6 )( 5,8 ),$$
    $$(1,9 )( 2,7 )( 3,10 )( 4,8 )( 5,6 ).$$

    \subsubsection{Results}

    \begin{theorem}
        Suppose $ \mathcal{B} $ is a 5-dimensional subfactorizable fusion bialgebra with exchange relations. Then $ \mathcal{B} $ arises from one of the following planar algebras:
        \begin{enumerate}
            \item[(1)] $ \mathcal{P}^{\mathbb{Z}_5} $;
            \item[(2)] $ \mathcal{P}*\mathcal{S} $ for two exchange relation planar algebras $ \mathcal{P} $ and $ \mathcal{S} $ with $ 2 \le \dim \mathcal{P}_{2,\pm} \le 4 $ and $ \dim \mathcal{P}_{2,\pm}+\dim \mathcal{S}_{2,\pm}=6 $.
        \end{enumerate}
    \end{theorem}

    \begin{proof}
        We follow our classification scheme from Step 3 to Step 7. 

        First, we list all admissible indicator functions (AIFs) in each case, the number of which depends on the number of equivalence classes shown above. Note that AIFs are just binary numbers. There are $ 2^{20} $, $ 2^{16} $, $ 2^{10} $ AIFs in cases (1), (2), (3) respectively. Then we check each AIF to detect whether its equivalent fusion graph is a forest. Put all forest fusion graphs in a set $ F_1 $. $ |F_1| = 47381, 4435, 137 $ in cases (1), (2), (3) respectively.
        
        Next, we partition $ F $ by orbits of group action. Select each element in an orbit to form a set $ F_2 $. $ |F_2| = 2137, 1292, 46 $ in cases (1), (2), (3) respectively. 

        And then, we apply APC to sieve $ F_2 $ to obtain a set $ F_3 $. $ |F_3| = 41, 18, 2 $ in cases (1), (2), (3) respectively.

        Next, we apply FTPC to sieve $ F_3 $ to obtain a set $ F_4 $. $ |F_4| = 0, 0, 1 $ in cases (1), (2), (3) respectively. 
        
        Finally, only one forest fusion graph $ \Gamma $ in case $ (3) $ passes all sieving criteria. We observe that each $ \Gamma_k $ is 1-regular, indicating that $ \Gamma $ is derived from a group planar algebra. Note that $ \mathbb{Z}_5 $ is the unique group of order $ 5 $ and it contains $ 2 $ pairs of dual minimal projections. Therefore, $ \Gamma $ originates from $ \mathcal{P}^{\mathbb{Z}_5} $. 

        Since $ 5 $ is a prime number, the tensor product criterion does not apply. Hence all other subfactorizable fusion bialgebras with exchange relations originate from a free product of exchange relation planar algebras.
    \end{proof}

    \begin{remark}
        Luckily, we achieve a classification result for 5-dimensional subfactorizable fusion bialgebra with exchange relations without the need to solve any polynomial equations directly.
    \end{remark}

    All data that appear in the proof are presented in Table \ref{table:5d}, which provides a comprehensive overview of the sieving process for subfactorizable fusion bialgebras with exchange relations, specifically for $ \dim \mathcal{P}_{2,\pm} = 5 $. The data were obtained by running our computer program \textbf{FuBi} on a 6-core AMD Ryzen 5600X @ 3.70 GHz with 16 GB of RAM using Microsoft Windows 10.

    We traverse over all admissible indicator functions (AIFs) to apply our sieving criteria from top to bottom in the table. The efficiency of the sieving process is particularly noteworthy given the large number of candidates involved. The forest structure of fusion graphs, denoted as \textbf{FF}, plays a crucial role in this efficiency, successfully eliminating 95.48\%, 93.23\% and 86.62\% AIFs in cases (1), (2), and (3), respectively. Additionally, the $ G $-action criterion \textbf{RG} effectively selects representative fusion graphs without duplicate. This process rules out 95.49\%, 70.87\%, and 66.42\% duplicate forest fusion graphs in cases (1), (2), and (3), respectively. The associative positivity criterion \textbf{APC} significantly reduces the number of forest fusion graphs. The sieving rate is 98.08\%, 98.61\%, and 95.65\% respectively.  The sieving rate of forest/tensor product criterion \textbf{FTPC} is 100\% in cases (1) and (2), indicating that every forest fusion graph derived from subfactor planar algebras originates from a free product of exchange relation planar algebras. 

    Ultimately, the application of these various criteria results in a unique forest fusion graph remaining, demonstrating the efficiency of the sieving process. Furthermore, the computational time required for this process is remarkably small, as detailed in the table below.

    \begin{table}[htbp]
        \centering
        \begin{tabular}{crrr}
        \toprule
        Criteria & \# AIFs(case 1)  & \# AIFs(case 2) & \# AIFs(case 3) \\
        \midrule
         & $ 1048576 $ & $ 65536 $ & $ 1024 $ \\
        \textbf{FF} & $ 47381 $ & $ 4435 $ & $ 137 $ \\
        \textbf{RG} & $ 2137 $ & $ 1292 $ & $ 46 $ \\
        \textbf{APC} & $ 41 $ & $ 18 $ & $ 2 $ \\
        \textbf{FTPC} & $ 0 $ & $ 0 $ & $ 1 $ \\
        \hline
        Time(seconds) & 9.443 & 1.264 & 0.162 \\
        \bottomrule
        \end{tabular}
        \caption{Sieving process for $ \dim \mathcal{P}_{2,\pm} = 5 $.}
        \label{table:5d}
    \end{table}

    This example illustrates that our classification scheme is highly efficient and practical. 
    While the computational complexity may pose challenges for larger numbers of generators due to the exponential growth of admissible indicator functions with respect to $ \dim \mathcal{P}_{2,+} $, our current scheme demonstrates a robust capability to handle the classification of 5-dimensional subfactorizable fusion bialgebras with exchange relations effectively.

    \subsection{Quick proof of previous work}\label{prevwork}
    Bisch, Jones, and the second author classified exchange relation planar algebras up to 4-dimension.
    We show tables of classification process for 4-dimensional and 3-dimensional fusion bialgebras to give a quick proof of their results.

    \begin{table}[htbp]
        \centering
        \begin{tabular}{crr}
        \toprule
        Criteria & \# Graphs(case 1)  & \# Graphs(case 2) \\
        \midrule
         & $ 1024 $ & $ 64 $  \\
        \textbf{FF} & $ 308 $ & $ 20 $  \\
        \textbf{RG} & $ 64 $ & $ 20 $  \\
        \textbf{APC} & $ 15 $ & $ 5 $  \\
        \textbf{FTPC} & $ 1 $ & $ 1 $  \\
        \hline
        Time(seconds) & 0.084 & 0.038  \\
        \bottomrule
        \end{tabular}
        \caption{Sieving process for $ \dim \mathcal{P}_{2,\pm} = 4 $.}
        \label{table:4d}
    \end{table}

    Table \ref{table:4d} shows the classification process for $ \mathcal{P}_{2,\pm} = 4 $. When the number of duals is 1 (case 2), only one forest fusion graph passes all criteria and all of its fusion graphs are 1-regular, so it originates from the group planar algebra $ \mathcal{P}^{\mathbb{Z}_4} $. When the number of duals is 0 (case 1), there are 1 remaining forest fusion graph. we demonstrate its fusion graphs of $ p_1, p_2, p_3 $.

    \begin{align}
        \Gamma_1: 
        \begin{tikzpicture}[scale=0.8,baseline={([yshift=-0.5ex]current bounding box.center)}]
            \foreach \x in {0,1,2,3}{
              \node[circle, fill=black, inner sep=2pt, outer sep=1pt, label=left: \x] (A\x) at (0,-\x) {};
            }
            \foreach \y in {0,1,2,3}{
              \node[circle,draw, inner sep=2pt, outer sep=1pt, label=right: \y] (B\y) at (2,-\y) {};
            }
            \draw (A0) -- (B1);
            \draw (A1) -- (B0);
            \draw (A1) -- (B2);
            \draw (A2) -- (B1);
            \draw (A2) -- (B3);
            \draw (A3) -- (B2);
            \draw (A3) -- (B3);
        \end{tikzpicture},
          \quad
        \Gamma_2: 
        \begin{tikzpicture}[scale=0.8,baseline={([yshift=-0.5ex]current bounding box.center)}]
            \foreach \x in {0,1,2,3}{
              \node[circle, fill=black, inner sep=2pt, outer sep=1pt, label=left: \x] (A\x) at (0,-\x) {};
            }
            \foreach \y in {0,1,2,3}{
              \node[circle,draw, inner sep=2pt, outer sep=1pt, label=right: \y] (B\y) at (2,-\y) {};
            }
            \draw (A0) -- (B2);
            \draw (A1) -- (B1);
            \draw (A1) -- (B3);
            \draw (A2) -- (B0);
            \draw (A2) -- (B3);
            \draw (A3) -- (B1);
            \draw (A3) -- (B2);
        \end{tikzpicture}, 
        \Gamma_3: 
        \begin{tikzpicture}[scale=0.8,baseline={([yshift=-0.5ex]current bounding box.center)}]
            \foreach \x in {0,1,2,3}{
              \node[circle, fill=black, inner sep=2pt, outer sep=1pt, label=left: \x] (A\x) at (0,-\x) {};
            }
            \foreach \y in {0,1,2,3}{
              \node[circle,draw, inner sep=2pt, outer sep=1pt, label=right: \y] (B\y) at (2,-\y) {};
            }
            \draw (A0) -- (B3);
            \draw (A1) -- (B2);
            \draw (A1) -- (B3);
            \draw (A2) -- (B1);
            \draw (A2) -- (B2);
            \draw (A3) -- (B0);
            \draw (A3) -- (B1);
        \end{tikzpicture}.
    \end{align}
    
    We solve its structure equations and conclude that this fusion graph originates from the group-subgroup subfactor planar algebra $ \mathcal{P}^{\mathbb{Z}_2\subset \mathbb{Z}_7\rtimes \mathbb{Z}_2} $.

    \begin{table}[htbp]
        \centering
        \begin{tabular}{crr}
        \toprule
        Criteria & \# Graphs(case 1)  & \# Graphs(case 2) \\
        \midrule
         & $ 16 $ & $ 4 $  \\
        \textbf{FF} & $ 10 $ & $ 2 $  \\
        \textbf{RG} & $ 6 $ & $ 2 $  \\
        \textbf{APC} & $ 5 $ & $ 1 $  \\
        \textbf{FTPC} & $ 1 $ & $ 1 $  \\
        \hline
        Time(seconds) & 0.01 & 0.006  \\
        \bottomrule
        \end{tabular}
        \caption{Sieving process for $ \dim \mathcal{P}_{2,\pm} = 3 $.}
        \label{table:3d}
    \end{table}

    Table \ref{table:3d} shows the sieving process for $ \mathcal{P}_{2,\pm} = 3 $. When the number of duals is 1 (case 2), the unique forest fusion graph originates from the group planar algebra $ \mathcal{P}^{\mathbb{Z}_3} $. When the number of duals is 0 (case 1), the fusion graphs of $ p_1 $ and $ p_2 $ of the remaining forest fusion graph are as follows.
    \begin{align}
        \Gamma_1: 
        \begin{tikzpicture}[scale=0.8,baseline={([yshift=-0.5ex]current bounding box.center)}]
            \foreach \x in {0,1,2}{
              \node[circle, fill=black, inner sep=2pt, outer sep=1pt, label=left: \x] (A\x) at (0,-\x) {};
            }
            \foreach \y in {0,1,2}{
              \node[circle,draw, inner sep=2pt, outer sep=1pt, label=right: \y] (B\y) at (2,-\y) {};
            }
            \draw (A0) -- (B1);
            \draw (A1) -- (B0);
            \draw (A1) -- (B2);
            \draw (A2) -- (B1);
            \draw (A2) -- (B2);
        \end{tikzpicture},
          \quad
        \Gamma_2: 
        \begin{tikzpicture}[scale=0.8,baseline={([yshift=-0.5ex]current bounding box.center)}]
            \foreach \x in {0,1,2}{
              \node[circle, fill=black, inner sep=2pt, outer sep=1pt, label=left: \x] (A\x) at (0,-\x) {};
            }
            \foreach \y in {0,1,2}{
              \node[circle,draw, inner sep=2pt, outer sep=1pt, label=right: \y] (B\y) at (2,-\y) {};
            }
            \draw (A0) -- (B2);
            \draw (A1) -- (B1);
            \draw (A1) -- (B2);
            \draw (A2) -- (B0);
            \draw (A2) -- (B1);
        \end{tikzpicture}.
    \end{align}
    We solve its structure equations and conclude that it originates from $ \mathcal{P}^{\mathbb{Z}_2\subset \mathbb{Z}_5\rtimes \mathbb{Z}_2} $.

    Note that exchange relation planar algebras of 2-dimensional 2-box space are precisely Temperley-Lieb-Jones planar algebras $ TL $ or $ \mathcal{P}^{\mathbb{Z}_2} $ in the reduced case. The Tensor Product construction gives planar algebras $ \mathcal{P}^{\mathbb{Z}_2}\otimes \mathcal{P}^{\mathbb{Z}_2} \cong \mathcal{P}^{\mathbb{Z}_2\oplus \mathbb{Z}_2} $ and $ \mathcal{P}^{\mathbb{Z}_2}\otimes TL $ in the 4-dimensional case. Thus we obtain the following theorems.

    \begin{theorem}[\cite{Bisch2000Singly,Bisch2003Singly}]\label{expa3d}
        Suppose $ \mathcal{P} $ is an exchange relation planar algebra with $ \dim \mathcal{P}_2=3 $. Then $ \mathcal{P} $ is one of the following: $ (1) \mathcal{P}^{\mathbb{Z}_3} $; $ (2) TL*TL $; $ (3) \mathcal{P}^{\mathbb{Z}_2\subset \mathbb{Z}_5\rtimes \mathbb{Z}_2} $. 
    \end{theorem}

    \begin{theorem}[\cite{Liu2016Exchange}]\label{expa4d}
        Suppose $ \mathcal{P} $ is an exchange relation planar algebra with $ \dim \mathcal{P}_2=4 $. Then $ \mathcal{P} $ is one of the following:
        \begin{enumerate}
            \item[(1)] $ \mathcal{P}^{\mathbb{Z}_4} $, or $ \mathcal{P}^{\mathbb{Z}_2\oplus \mathbb{Z}_2} $;
            \item[(2)] $ \mathcal{A}*TL $, or $ TL*\mathcal{A} $, for an exchange relation planar algebra $ \mathcal{A} $ with $ \dim \mathcal{A}_2=3 $;
            \item[(3)] $ \mathcal{P}^{\mathbb{Z}_2}\otimes TL $;
            \item[(4)] $ \mathcal{P}^{\mathbb{Z}_2\subset \mathbb{Z}_7\rtimes \mathbb{Z}_2} $.
        \end{enumerate}
    \end{theorem}

    \section{Outlook}

    In this section, we outline several open questions that pave the way for future research on our classification program for subfactor planar algebras. These questions aim to deepen our understanding of the structure and classification of these planar algebras.

    First, we list some questions about exchange relations. The basic operations in planar algebras include dual operations, free products and tensor products. The second author has proved in \cite{Liu2016Exchange} that the free product of two exchange relation planar algebras is an exchange relation planar algebra, the tensor product of a depth-2 subfactor planar algebra and an exchange relation planar algebra also yields an exchange relation planar algebra, as does the dual of an exchange relation planar algebra.

    \begin{question}
        Is there a complete classification of abelian subfactorizable fusion bialgebras with exchange relations?
    \end{question}    
    This question is inspired by a classical structure theorem of finite abelian groups, which states that every finite abelian group can be expressed as a direct sum of cyclic groups of prime-power order. These cyclic groups serve as the fundamental building blocks of abelian groups. In a parallel manner, we seek to identify the basic components of exchange relation planar algebras, ensuring that these components do not arise from free or tensor products. Currently, we have identified several types of planar algebras, including Temperley-Lieb-Jones planar algebras $ TL(\delta) $, group planar algebras $ \mathcal{P}^{G} $, group-subgroup planar algebras $ \mathcal{P}^{\mathbb{Z}_2\subset \mathbb{Z}_p\rtimes \mathbb{Z}_2} $ where $ p $ is an odd prime. A critical question remains: are these basic components finite? If they are, this could lead us toward a complete classification.

    \begin{question}
        What is the classification of subfactorizable fusion bialgebras with exchange relations?
    \end{question}

    The forest fusion graph reveals a certain rigidity in exchange relation planar algebras. This case includes all finite dimensional group algebras. Hence, we can only expect a complete classification up to finite groups. However, our immediate priority is to develop efficient algorithms that can extend the classification to include more generators and discover new examples. Each new example will provide valuable insights and experience that will enhance our understanding of the underlying structures.

    \begin{question}
        What is the classification of exchange relation planar algebras for $ \dim \mathcal{P}_{2,\pm} = 5 $ when $ \mathcal{P}_{2,\pm} $ are non-abelian?
    \end{question}

    When \( \mathcal{P}_{2,+} \) is abelian, it can be represented as a function algebra over a finite set of points, allowing for the selection of a canonical basis that simplifies the actions of planar operations, such as \( 90^{\circ} \) rotation, \( 180^{\circ} \) rotation, and vertical reflection. In contrast, if \( \mathcal{P}_{2,\pm} \) are non-abelian, they are both isomorphic to \( \mathbb{C} \oplus M_2(\mathbb{C}) \). In this case, the actions of planar operations are not uniquely defined and are instead characterized by continuous parameters. While we have made initial attempts to analyze these structures, we are still far from achieving a complete classification.

    Nevertheless, the 8-dimensional Kac algebra $ K_8 $ serves as an example where both multiplication and convolution are non-abelian.
    Since exchange relation planar algebras generalize finite dimensional Kac algebras, We pose the following question:
    \begin{question}
        What is the smallest dimension of 2-boxes in exchange relation planar algebras such that $ \mathcal{P}_{2,\pm} $ are non-abelian?
    \end{question}

    From a skein theoretical perspective, we naturally consider more comprehensive skein relations, specifically the Yang-Baxter relations.
    \begin{question}
        What are consistency equations of Yang-Baxter relations?
    \end{question}
    A planar algebra satisfies Yang-Baxter relations if and only if the dimension of 3-box space matches the number of non-zero fusion coefficients of a full type fusion bialgebra, see \cite{Liu2016Yang}. Further exploration in \cite{Liu2017generator} revealed that the even part of the Haagerup subfactor planar algebra $ H \mathcal{P} $ satisfies Yang-Baxter relations, however, the precise calculation of the parameters within Yang-Baxter relations remains unresolved. This motivates to derive complete consistency equations of Yang-Baxter relations. If one could solve the consistency equations when the fusion bialgebra is 4-dimensional, one would achieve a classification including $ H \mathcal{P} $. 
    \begin{conjecture}
        The even part of the Haagerup subfactor planar algebra $ H \mathcal{P} $ lies in a one-parameter family of subfactor planar algebras.
    \end{conjecture}

    In summary, these open questions and conjectures highlight the potential for significant advancements in our understanding of subfactor planar algebras. We encourage further exploration in these areas, as they promise to yield valuable insights into the intricate structures and classifications of these mathematical objects.

    \bibliography{cerpa}

\begin{thebibliography}{LPW21}

\bibitem[BJ97]{Bisch1997Algebras}
Dietmar Bisch and Vaughan Jones.
\newblock Algebras associated to intermediate subfactors.
\newblock {\em Invent. Math.}, 128(1):89--157, 1997.

\bibitem[BJ00]{Bisch2000Singly}
Dietmar Bisch and Vaughan Jones.
\newblock Singly generated planar algebras of small dimension.
\newblock {\em Duke Math. J.}, 101(1):41--75, 2000.

\bibitem[BJ03]{Bisch2003Singly}
Dietmar Bisch and Vaughan Jones.
\newblock Singly generated planar algebras of small dimension. {II}.
\newblock {\em Adv. Math.}, 175(2):297--318, 2003.

\bibitem[BW89]{Birman1989Braids}
Joan~S. Birman and Hans Wenzl.
\newblock Braids, link polynomials and a new algebra.
\newblock {\em Trans. Amer. Math. Soc.}, 313(1):249--273, 1989.

\bibitem[Jon83]{Jones1983Index}
V.~F.~R. Jones.
\newblock Index for subfactors.
\newblock {\em Invent. Math.}, 72(1):1--25, 1983.

\bibitem[Jon99]{Jones1999Planar}
Vaughan F.~R. Jones.
\newblock Planar algebras, {I}, 1999.
\newblock arXiv:math/9909027.

\bibitem[Jon12]{Jones2012Quadratic}
Vaughan F.~R. Jones.
\newblock Quadratic tangles in planar algebras.
\newblock {\em Duke Math. J.}, 161(12):2257--2295, 2012.

\bibitem[Lan02]{Landau2002Exchange}
Zeph~A. Landau.
\newblock Exchange relation planar algebras.
\newblock In {\em Proceedings of the {C}onference on {G}eometric and
  {C}ombinatorial {G}roup {T}heory, {P}art {II} ({H}aifa, 2000)}, volume~95,
  pages 183--214, 2002.

\bibitem[Liu16a]{Liu2016Exchange}
Zhengwei Liu.
\newblock Exchange relation planar algebras of small rank.
\newblock {\em Trans. Amer. Math. Soc.}, 368(12):8303--8348, 2016.

\bibitem[Liu16b]{Liu2016Yang}
Zhengwei Liu.
\newblock Yang-baxter relation planar algebras, 2016.
\newblock arXiv:1507.06030.

\bibitem[LP17]{Liu2017generator}
Zhengwei Liu and David Penneys.
\newblock The generator conjecture for {$3^G$} subfactor planar algebras.
\newblock In {\em Proceedings of the 2014 {M}aui and 2015 {Q}inhuangdao
  conferences in honour of {V}aughan {F}. {R}. {J}ones' 60th birthday},
  volume~46 of {\em Proc. Centre Math. Appl. Austral. Nat. Univ.}, pages
  344--366. Austral. Nat. Univ., Canberra, 2017.

\bibitem[LPW21]{Liu2021Fusion}
Zhengwei Liu, Sebastien Palcoux, and Jinsong Wu.
\newblock Fusion bialgebras and {F}ourier analysis: analytic obstructions for
  unitary categorification.
\newblock {\em Adv. Math.}, 390:Paper No. 107905, 63, 2021.

\bibitem[Mur87]{Murakami1987Kauffman}
Jun Murakami.
\newblock The {K}auffman polynomial of links and representation theory.
\newblock {\em Osaka J. Math.}, 24(4):745--758, 1987.

\bibitem[Wen90]{Wenzl1990Quantum}
Hans Wenzl.
\newblock Quantum groups and subfactors of type {$B$}, {$C$}, and {$D$}.
\newblock {\em Comm. Math. Phys.}, 133(2):383--432, 1990.

\end{thebibliography}

\end{document}